\documentclass[11pt]{article}
\usepackage{amsmath,amsxtra,amssymb,color,latexsym,epsfig,amscd,amsthm,subfigure,fancybox,epsfig}
\usepackage[mathscr]{eucal}
\usepackage{bbm}
\usepackage{float}
\usepackage{graphicx}
\usepackage{enumerate}
\usepackage{enumitem}
\usepackage{epsfig}
\usepackage{epstopdf}
\usepackage{cases}
\usepackage{hyperref}
\def\tr{{\rm tr}}

\newtheorem{thm}{Theorem}[section]
\newtheorem {asp}{Assumption}[section]
\newtheorem{lm}{Lemma}[section]

\newtheorem{deff}{Definition}[section]

\theoremstyle{definition}

\theoremstyle{remark}
\newtheorem{rem}{Remark}[section]
\newtheorem{exam}{Example}[section]
\numberwithin{equation}{section}

\DeclareMathOperator{\trace}{tr}

\newcommand{\eps}{\varepsilon}

\newcommand{\h}{\mathcal{H}}
\newcommand{\A}{\mathcal{A}}

\newcommand{\D}{\mathcal{D}}

\newcommand{\E}{\mathbb{E}}

\newcommand{\Lom}{\mathcal{L}}

\newcommand{\Q}{{\mathbb{Q}}}
\newcommand{\bx}{{\mathbf{x}}}
\newcommand{\by}{{\mathbf{y}}}
\newcommand{\bz}{{\mathbf{z}}}

\newcommand{\wN}{{\widetilde{\mathbf{N}}}}

\newcommand{\R}{\mathbb{R}}
\newcommand{\Ra}{{\R^{n_1}_*}}
\newcommand{\Rb}{{\R^{n_2}_*}}
\newcommand{\BX}{\mathbf{X}}
\newcommand{\BY}{\mathbf{Y}}
\newcommand{\BZ}{\mathbf{Z}}
\newcommand{\BW}{\mathbf{W}}

\newcommand{\Sc}{\mathcal{S}}

\newcommand{\PP}{\mathbb{P}}

\newcommand{\RX}{\mathbb{R}^{\ell_1}}
\newcommand{\RY}{\mathbb{R}^{\ell_2}}

\newcommand{\norm}[1]{\left\Vert#1\right\Vert}

\numberwithin{equation}{section}
\newcommand{\1}{\boldsymbol{1}}

\newcommand{\bphi}{\boldsymbol{\phi}}
\newcommand{\0}{\boldsymbol{0}}

\newcommand{\wdt}{\widetilde}

\newcommand{\op}{{\cal L}}

\newcommand{\bed}{\begin{equation}}
\newcommand{\eed}{\end{equation}}
\newcommand{\bea}{\bed\begin{array}{rl}}
\newcommand{\eea}{\end{array}\eed}
\newcommand{\ad}{&\!\!\!\disp}
\newcommand{\aad}{&\disp}
\newcommand{\barray}{\begin{array}{ll}}
\newcommand{\earray}{\end{array}}
\newcommand{\diag}{{\rm diag}}

\def\disp{\displaystyle}

\def\bar{\overline}
\def\hat{\widehat}
\def\a.s{\text{\;a.s.\;}}

\def\bnu{\boldsymbol{\nu}}

\def\deltaxe{{\theta_{\bx_1,\eps}}}
	\def\txe{{T_{\bx_1,\eps}}}
		\def\mxe{{M_{\bx_1,\eps}}}

\begin{document}
\title{Stability and Stabilization of Coupled Jump Diffusions and Applications}
\author{Dang H. Nguyen\thanks{Department of Mathematics, University of Alabama, Tuscaloosa, AL
35487, dangnh.maths@gmail.com. The research of this author was supported in part by the National Science Foundation under grant DMS-1853467.},
\and Duy Nguyen\thanks{Department of Mathematics, Marist College, Poughkeepsie, NY 12601, nducduy@gmail.com.},
 \and Nhu N. Nguyen\thanks{Department of Mathematics, University of Connecticut, Storrs, CT
06269, nguyen.nhu@uconn.edu. The research of this author was supported in part by the Air Force Office of Scientific Research.}, \and George Yin\thanks{Department of Mathematics, University of Connecticut, Storrs, CT
06269, gyin@uconn.edu. The research of this author was supported in part by the Air Force Office of Scientific Research.}}
\date{}
\maketitle

\begin{abstract}
This paper develops
stability and stabilization results for systems of fully coupled jump diffusions. Such  systems  frequently arise
in  numerous applications where each subsystem (component) is
operated under the influence of
other subsystems (components). This paper derives sufficient conditions under which the underlying
coupled jump diffusion is stable.
The results are then applied to investigate the stability of  linearizable jump diffusions, fast-slow coupled jump diffusions.
Moreover, 
weak stabilization of interacting systems and consensus of leader-following systems are examined.

\bigskip


\end{abstract}

\newpage
\section{Introduction}\label{sec:int}
Nowadays, networked systems have posed unprecedented opportunities as well as challenges. Such systems have
numerous
applications in control
engineering, wireless communications, mathematical biology, financial engineering,
and actuarial science.
In many
stochastic  networked systems, subsystems and/or components are intertwined and highly coupled. Moreover, 
empirical studies reveal that there exists sudden rapid moments in the mid quotes of
stock prices, i.e., jumps during trading periods
\cite{BatesJumps,LeeDetect}.
Treating competitive Lotka-Volterra populations, as observed in \cite{BMYY},
 the population may suffer sudden environmental shocks such as earthquakes, hurricanes, epidemics, etc. The commonly used diffusion-type  stochastic Lotka-Volterra models  cannot explain such phenomena.
To allow sudden changes, systems of
stochastic differential equations
involving  L\'{e}vy process are often used
to capture
fluctuations as well as
jumps caused by rare events.

When a system has been operated for a long time, its long-time behavior becomes an important feature.
Stability of such
systems has therefore been studied extensively.
Given a system with coupled components/subsystems, can we derive the stability of one specific component based on the dynamics of the other components?
Take
for instance, a system involving fast-slow components,
one uses different time scales to portrait the fast-slow motions.
 A question is almost immediate. How can we determine the stability of the slow subsystem based on the
information of
the fast subsystem?
In addition, for a coupled system of jump diffusions, can we design a feedback control so as to obtain the desired stability?
This work aims to examine these questions.


Previous works on
stability of jump-diffusion processes can be found in
\cite{Wee99} for multidimensional jump-diffusion processes, \cite{AB02} for constrained jump-diffusion processes,
\cite{FTT10} for jump diffusions in a Hilbert space, \cite{CWZZ17,CCTY19,CCTY19a,YX10,ZWYJ14} for regime-switching jump diffusions, \cite{BGR18} for jump diffusions with state-dependent densities.
In contrast to the existing works in the literature,
 this work focuses on stability of fully coupled jump diffusions, where two jump-diffusion components interact with each other.
Such
systems have a
wide range of applications to numerous physical, engineering, and biological problems such as chemical
reactors \cite{MADF10}, power
transmission lines \cite{DTV14}, flow regulation in deep mines \cite{WD10}, elastic beams linked to rigid
bodies \cite{Lit98}, blood flow model \cite{CZ04,FM05}, mitochondrial swelling \cite{EOE20},  to mention  just a few
among others. In such a situation, we are  interested in the stability, averaging phenomena
under the
influence of the
interacting processes in the environment.
From a technical point of view, not only do
the coupled systems  possess many interesting properties,
 but also present many challenges;
  see coupled ordinary differential and partial differential equations (ODEs-PDEs) \cite{AA13,KT16,LMP11,R08}, coupled diffusion or stochastic differential equations \cite{HL20,HL16,Puh16}, and coupled stochastic reaction-diffusion or stochastic partial differential equations \cite{Cer11,CF09,KW13}.
The motivations and urgent need
in both theory and applications motivate the current work.

Let $\RX$ and $\RY$ be two Euclidean spaces of dimensions $\ell_1>0$ and $\ell_2>0$, respectively. Denote by $\0$ a zero vector with appropriate dimension (which will be clear from the context).
We assume $\BX_1(t)$ and
$\BX_2(t)$ are coupled
jump-diffusion processes in $\RX$, $\RY$, respectively.
More specifically, the pair $(\BX_1(t), \BX_2(t))$ is the solution of
 the system
\begin{equation}\label{e2.3}
\begin{cases}
\displaystyle d\BX_1(t)=b_1(\BX_1(t),\BX_2(t))dt+\sigma_1(\BX_1(t),\BX_2(t))d\BW_1(t)+\int_\Ra \gamma_1(\BX_1(t-),\BX_2(t-), \bphi)\wN_1(dt,d\bphi), \\
\displaystyle d\BX_2(t)=b_2(\BX_1(t),\BX_2(t))dt+\sigma_2(\BX_1(t),\BX_2(t))d\BW_2(t)+\int_\Rb \gamma_2(\BX_1(t-),\BX_2(t-), \bphi)\wN_2(dt,d\bphi),\\
\BX_1(0)=\bx_1,\quad \BX_2(0)=\bx_2,
\end{cases}
\end{equation}
where,
for $i=1,2$, $\BW_i(t)$ are standard $\R^{d_i}$-valued Brownian motions; $\bold N_i(dt, d\bphi)$'s
are Poisson random measures independent
of $\BW_1(t)$, $\BW_2(t)$, and $\wN_i(dt, d\bphi) = \bold N_i(dt, d\bphi) - \bnu_i(d\bphi)dt$ are the compensated Poisson random measures on
$[0, \infty) \times \R_*^{n_i}$ with $\mathbb R_*^{n_i}:=\mathbb R^{n_i}\setminus\{\mathbf 0\}$;
$b_1:\RX\times\RY\to\RX$ and
$b_2:\RX\times\RY\to\RY$ are smooth functions;
$\sigma_1: \RX_1 \times \RY_2 \to \R^{\ell_1 \times d_1}$
and $\sigma_2: \RX_1 \times \RY_2 \to \R^{\ell_2 \times d_2}$;
$\gamma_1:\RX\times\RY\times\R_*^{n_1}\to\RX$,
$\gamma_2:\RX\times\RY\times\R_*^{n_2}\to\RY$ are measurable functions.
In this paper, it is assumed that
$$
b_2(\bx_1,\0)=\0,\ \sigma_2(\bx_1,\0)=\0,\  \gamma_2(\bx_1,\0,\bphi)=\0,
$$
and thus $\0$ is an equilibrium point of $\BX_2(t)$.
We wish to derive conditions that are
mild
 under which the equilibrium point $\0$ or the
 trivial solution of
  $\BX_2(t)$ is stable.

In what follows,
the assumptions for the stability are given in term of Lyapunov functions but the insight and intuition are derived
from a dynamic system point of view combined with the averaging principle and ergodicity of coupled systems.
We demonstrate that the conditions are easily
applicable.
The intuition and idea are that if the 
interacting 
process $\BX_1(t)$ on the boundary (i.e., when $\BX_2(t)=\0$) admits a unique invariant measure and the corresponding decoupled (averaging) equations (of the process $\BX_2(t)$) at the stationary distribution of $\BX_1(t)$ satisfies some appropriate stability conditions then $\BX_2(t)$ is also stable.
Taking the idea from a dynamic system point of view, the stability conditions are obtained by considering the Lyapunov exponent of the process $\BX_2(t)$ corresponding to the invariant measure of the interacting process $\BX_1(t)$ on the boundary. 
Such conditions coincide with the intuition that when the main process $\BX_2(t)$ is close to the equilibrium $\0$, the interacting process is close to the solution on the boundary.
 Thus the stability conditions of the main process only need to be based on the information of the interacting process on the boundary, which
however,
poses great challenges. We need to reveal the behavior of the system around the boundary. Since two components are fully coupled,
handling their interactions and analyzing their behavior
require a careful analysis.
To overcome the difficulties, one of the main tools  used in this paper
is the coupling method \cite{CL89}.

With the
stability results at our hands, we study 
systems that commonly arise in applications.
In particular,
one of the main questions one would like to answer is: What are the relationships between nonlinear systems and the associated
linearized systems for jump diffusions? 
We address this question and provide  sufficient conditions for the stability of linearizable jump diffusions.
In various applications, the subsystems and/or components often display different time-scales. It is often necessary to treat
fast-slow coupled jump diffusions.
We provide sufficient conditions for  stability of the slow component based only on the limit system.
Next, we design stabilizing strategies in a coupled system when only the interacting process is available to be controlled.
Finally, 
leader-following systems are studied and 
conditions for the consensus controllability of the underlying systems are given.

The rest of paper is arranged as follows. Section \ref{sec:main} presents the main results on stability. 
Section \ref{sec:app1} deals with linearizable systems and systems with fast and slow components. Section \ref{sec:stabl} focuses on stabilization and treats consensus problems. Finally,
Section \ref{sec:con} concludes the paper and issues
further remarks.

\section{Stability of Coupled Jump Diffusions 
}\label{sec:main}
In this paper, we use $|\cdot|$ to denote the Euclidean norm for either vectors or matrices, 
and
$A^\top$ the transpose of a vector or a matrix $A$. For two real numbers
$a$ and $b$,  $a\vee b= \max(a, b)$, and $a\wedge b = \min(a, b)$.
 Denote $\BZ=(\BX_1^\top,\BX_2^\top)^\top$, $\bz=(\bx_1^\top,\bx_2^\top)^\top$, $b=(b_1^\top,b_2^\top)^\top$, $\sigma=(\sigma_1^\top,\sigma_2^\top)^\top$, $\gamma=(\gamma_1^\top,\gamma_2^\top)^\top$, $\BW=(\BW_1^\top,\BW_2^\top)^\top$, and $\wN=(\wN_1^\top,\wN_2^\top)^\top$.
 We will use $\BZ$ and $(\BX_1,\BX_2)$ exchangeably.  Moreover, a function of $(\bx_1,\bx_2)$ can often be written as a function of $\bz$ with $\bz=(\bx_1^\top,\bx_2^\top)^\top$, which will be clear from the context.
Note that the equation \eqref{e2.3}
 in vector form becomes
$$
d\BZ(t)=b(\BZ(t))dt+\sigma(\BZ(t))d\BW(t)+\int_{\Ra\times\Rb} \gamma(\BZ(t-), \bphi)\wN(dt,d\bphi),\quad
\BZ(0)=\bz.
$$
We will use $\PP_{\bz}$ (and/or $\E_{\bz}$) to indicate the initial conditions of the system.

Next, define the operator $\op$ by
$$
\begin{aligned}
\op g(\bz):=&[\partial_\bz g(\bz)]^\top b(\bz)+\frac 12\text{tr}[\sigma(\bz)\sigma^\top(\bz)\partial^2_{\bz}g(\bz)]\\
&+\int_{\Ra\times\Rb}[g(\bz+\gamma(\bz,\bphi))-g(\bz)-(\partial_\bz g)^\top\gamma(\bz,\bphi)]\bnu(d\bphi),
\end{aligned}
$$
for $g\in\D_\op$, where $\partial_\bz$ and $\partial^2_{\bz}$ denote the gradient and Hessian matrix with respect to $\bz$, respectively, and
$$
\begin{aligned}
\ad \D_{\op}:=
\big\{g:\RX\times\RY\to\R: g(\bz)
\text{ is twice continuously differentiable}
 \text{ and }\\
\aad
\int_{\Ra\times\Rb}|g(\bz+\gamma(\bz,\bphi))-g(\bz)-\partial_\bz g\cdot\gamma(\bz,\bphi)|\nu(d\bphi)<\infty
\big\}.
\end{aligned}
$$
It is noted that $\bz$ in $\op g(\bz)$ represents the variable of $\op g$ rather than the variable of $g$. Indeed, later $g$ can be plugged in by either functions of $\bx_1$ or functions of $\bx_2$.
For example, if $g$ is a function of $\bx_1$ only, the gradient of $g$ (with respect to $\bz$) will be $([\partial_{\bx_1}g(\bx_1)]^\top,\0^\top)^\top$. However, because the coefficients are fully coupled, $\Lom g$ is
a function of $\bz$;  we still write it as $\Lom g(\bz)$.
[This example is in fact associated to the generalized It\^o formula (given
below) for the first component.]


The following result is known as the generalized It\^o formula (see e.g., \cite{Sko89,YX10})
$$
\begin{aligned}
g(\BZ(t))-g(\BZ(0))=&\int_0^t \op g(\BZ(s-))ds+\int_0^t \partial_{\bz}g(\BZ(s-))\sigma(\BZ(s-))d\BW(s)\\
&+\int_0^t\int_{\Ra\times\Rb}\left[g(\BZ(s-)+\gamma(\BZ(s-),\bphi))-g(\BZ(s-))\right]\wN(ds,d\bphi).
\end{aligned}
$$
To ensure the existence and uniqueness of the solution, we impose the following assumption.

\begin{asp}\label{asp1} {\rm There are some constants $K_1, K_2>0$ such that $\forall \bz_1,\bz_2\in\RX\times\RY$
\bed\label{a.eq1}
|b(\bz_1)-b(\bz_2)|^2+|\sigma(\bz_1)-\sigma(\bz_2)|^2+\int_{\Ra\times\Rb}|\gamma(\bz_1,\bphi)-\gamma(\bz_2,\bphi)|^2\bnu(d\bphi)\leq K_1|\bz_1-\bz_2|^2,
\eed
and
\bed\label{a.eq2}
\int_{\Ra\times\Rb}|\gamma(\bz,\bphi)|^2\bnu(d\bphi)\leq K_2(1+|\bz|^2).
\eed
}\end{asp}

\begin{rem}
	Assumption \ref{asp1} can be replaced
with local Lipschitz conditions together with a suitable condition imposed on a Lyapunov function.
\end{rem}

To investigate the stability of the coupled jump diffusion system, we need the following assumptions.

\begin{asp}\label{asp2}
{\rm The following conditions hold.
\begin{itemize}
\item[]{\rm (i)}
	There exist positive functions $V_0, V_1: \RX\mapsto\R_+$ satisfying
\bed\label{a.eq3}
\op V_0(\bz)\leq K_3- K_4V_1(\bx_1),\;\forall \bz=(\bx_1,\0),
\eed
for some constants $K_3,K_4>0$.
\item[]{\rm (ii)}
There exist a function $U:\RY\mapsto\R_+$ and a constant $m_0>0$ such that
\bed\label{a.eq5}
\lim_{\bx_2\to\0} U(\bx_2)=\infty,\quad
U(\bx_2) -U(\bx_2')\leq m_0\ln\frac{|\bx_2'|}{|\bx_2|},\;\forall \bx_2,\bx_2'\neq\0.
\eed
Moreover, assume that there are a function $f_1:\RX\to\R$, and a Lipschitz function $f_2:\RX\to\R_+$, constants $\Delta_0>0$, $\alpha_0>0$,  and $K_5>0$ such that
\bed\label{a.eq6}
\op U(\bz)\geq f_1(\bx_1),\forall \bz=(\bx_1,\bx_2), |\bx_2|\leq \Delta_0,
\eed
and 
\bed\label{a.eq8}
\begin{aligned}
|U_{\bx_2}\sigma_2(\bz)|^2&+\int_\Rb\Big[\exp\Big\{-\alpha_0\big(U(\bx_2+\gamma_2(\bz,\bphi))-U(\bx_2)\big)_+\Big\}\Big]\bnu_2(d\bphi)\\
&\leq f_2(\bx_1),\;\forall \bz=(\bx_1,\bx_2), |\bx_2|\leq \Delta_0,
\end{aligned}
\eed
where $U_{\bx_2}$ denotes the gradient of $U$;
and
\bed\label{a.eq9}
|f_1(\bx_1)|+f_2(\bx_1)<K_5 V_1(\bx_1).
\eed

	\item[]{\rm (iii)} When $\bx_2=\0$ (yielding $\BX_2(t)=\0$),  the corresponding system for $\BX_1$
	$$
	d\BX_1(t)=b_1(\BX_1(t),\0)dt+\sigma_1(\BX_1(t),\0)d\BW_1(t)+\int_\Ra \gamma_1(\BX_1(t-),\0, \bphi)\wN_1(dt,d\bphi),
	$$
	 admits a unique invariant measure $\mu^*$ and
	$$
	\Lambda_1:=\int f_1(\bx_1)\mu^*(d\bx_1)>0.
	$$
	
\end{itemize}
}
\end{asp}
\begin{asp}\label{asp4}
{\rm Suppose that $\sigma_1(\bx_1,\bx_2)$ admits a right inverse $\sigma_1^{-1}(\bx_1,\bx_2)$ such that	$$\|\sigma_1^{-1}(\bx_1,\bx_2)\|\leq c_\sigma<\infty\text{ for all }|\bx_2|\leq\Delta_0,$$
	where $\Delta_0$ is as in Assumption \ref{asp2}.
Moreover, assume that	
there is a constant $K_6$ such that $\forall \bz=(\bx_1^\top,\bx_2^\top)^\top,\bz'=((\bx_1')^\top,(\bx_2')^\top)^\top\in\RX\times\RY$
	$$(b_1(\bz)-b_1(\bz'))^\top(\bx_1-\bx_1') +|\sigma_1(\bz)-\sigma_1(\bz')|^2\leq K_6|\bz-\bz'|^2.$$
In addition,
	$\gamma_1$ and $\gamma_2$ are Lipschitz in $\bz$ and uniformly in $\bphi$.}
\end{asp}

\begin{rem}
	Assumption \ref{asp2} is the main assumption for stability. This assumption is rather mild and
not restrictive.
	The functions $U$ and $f_1$ in Assumption \ref{asp2}(ii) are used to estimate the Lyapunov exponent of $\BX_2(t)$; $\Lambda_1$ in Assumption \ref{asp2}(iii) is a bound of the Lyapunov exponent.
	A simple but promising candidate of
	$U(\cdot)$ that satisfies the proposed conditions is $U(\bx_2)=(-\ln |\bx_2|)\vee 0.$
	Assumption \ref{asp2}(i) and \eqref{a.eq9} guarantee that the bound $\Lambda_1$ of the Lyapunov exponent is well-defined.
	Assumption \ref{asp4} collects a
	  form
	  of strong non-degenerate condition of the diffusion and some technical conditions.
	It can be seen that
	the conditions are applicable to
	many systems in applications;
		see Example \ref{ex-1} below as well as Section \ref{sec:app1}.

\end{rem}

Now, we state our main results.

\begin{thm}\label{thm-main}
	For any $\bx_1\in\RX$ and $\eps>0$, there exists $\deltaxe>0$ such that
	if $|\bx_1-\wdt\bx_1|+|\wdt\bx_2|\leq\deltaxe$,
	\begin{equation}\label{eq-thm-1}
	\PP_{\wdt\bz}\Big\{\liminf_{t\to\infty}\frac{U(\BX_2(t))}t\geq m_0\gamma_0\Big\}\geq 1-\eps,
	\end{equation}
	where $\wdt\bz=(\wdt\bx_1^\top,\wdt\bx_2^\top)^\top$; and thus,
	\begin{equation}\label{eq-thm-2}
	\PP_{\wdt\bz}\Big\{\limsup_{t\to\infty}\frac{\ln|\BX_2|}t\leq -\gamma_0\Big\}\geq 1-\eps.
	\end{equation}
\end{thm}

\begin{exam}\label{ex-1}
	To illustrate our results, let us provide a simple example. Consider a stochastic SIR epidemic model with Beddington-DeAngelis incidence rate given by the following SDEs with jumps,
	\begin{equation}
	\begin{cases}
	dS(t)=\Big[c_0-c_1S(t)- \frac{c_3S(t)I(t)}{c_4+c_5S(t)+c_6I(t)}\Big]dt+ \sigma_1(S(t),I(t))dW_1(t)+\disp\int_\Ra \!\!
\gamma_1(S(t),I(t), \bphi)\wN_1(dt,d\bphi),\\[1ex]
	dI(t)=\Big[-c_2I(t)+\frac{c_3S(t)I(t)}{c_4+c_5S(t)+c_6I(t)}\Big]dt
+c_7I(t)dW_2(t)+I(t)\disp\int_\Ra \hat\gamma_2(\bphi)\wN_1(dt,d\bphi).
	\end{cases}
	\end{equation}
	In the above, $\sigma_1,\gamma_1,\hat\gamma_2$ are bounded functions and are such that $\sigma_1(s,i),\gamma_1(s,i,\bphi),\gamma_2(s,i,\bphi)=i\hat\gamma_2(\bphi)$ satisfy the technical conditions in Assumptions \ref{asp1} and \ref{asp4}.
	Now, we check the stability conditions (Assumption \ref{asp2}).
	It is easily verified that Assumption \ref{asp2}(i) is satisfied with $V_0(s)=V_1(s)=s$; Assumption \ref{asp2}(ii) is satisfied with
	$$U(i)=(-\ln i)\vee 0, \quad f_1(s)=c_2+\frac{c_7^2}2+\int_{\Rb} |\hat\gamma_2(\phi)|^2\bnu_2(d\bphi)-\frac{c_3s}{c_4+c_5s},$$
	and $f_2(s)$ is some large constant. Under certain conditions, the corresponding system when $I(t)=0$
	$$
	d\hat S(t)=(c_0-c_1\hat S(t))dt+\sigma_1(\hat S(t),0)dW_1(t)+\int_\Ra \gamma_1(\hat S(t),0, \bphi)\wN_1(dt,d\bphi)
	$$
	has a unique invariant measure $\mu^*$.
	Therefore,
	if
	$\Lambda_1:=\int f_1(s)\mu^*(ds)>0$, applying Theorem \ref{thm-main}, $I(t)$ is stable at $0$. Without jump, this result is consistent with the longtime characterization in stochastic SIR epidemic models in \cite{DDN19,DN20};
 it generalizes the results in the aforementioned references.
\end{exam}

To prove Theorem \ref{thm-main}, we begin with some
auxiliary lemmas.
Lemma \ref{lm1} provides a local boundedness (uniform in finite intervals) in probability of the solution and Lemma \ref{lm2} illustrates the continuity on initial value (at $\0$) in probability of $\sup_{t\in[0,T]}|\BX_2(t)|$, for any finite time $T$.

\begin{lm}\label{lm1}
	For any $T>0$, $\eps>0$, $R>0$, there exists $H_{1,T,\eps, R}>0$ such that
	$$\PP_\bz\Big\{\sup_{t\in[0,T]}|\BZ(t)|\leq H_{1,T,\eps, R}\Big\}>1-\eps,\text{ for all } |\bz|\leq R.$$
\end{lm}

\begin{proof}
Under Assumption \ref{asp1}, using a standard argument (see e.g., \cite[Lemma 6.9]{Mao97}),
we
obtain the following local boundedness
$$
\E_\bz\sup_{t\in[0,T]}|\BZ(t)|<C_T(\bz),
$$
where $C_T(\bz)$ is some finite constant depending on $T$  and $\bz$ that is locally bounded in $\bz$.
As a result, the Markov inequality implies that for any $H>0$,
$$\PP_\bz\Big\{\sup_{t\in[0,T]}|\BZ(t)|> H\Big\}\leq \frac{ \disp\E_\bz\sup_{t\in[0,T]}|\BZ(t)|}{H}<\frac{C_T(\bz)}{H},
$$
which yields that for any $T>0$, $\eps>0$, and $R>0$, there is an $H_{1,T,\eps,R}>0$ such that
$$\PP_\bz\Big\{\sup_{t\in[0,T]}|\BZ(t)|\leq H_{1,T,\eps, R}\Big\}>1-\eps,\text{ for all } |\bz|\leq R.$$
The proof is complete.
\end{proof}

\begin{lm}\label{lm2}
	For any $T>0$, $\eps>0$, $R>0$, and $\delta_1>0$, there exists $\delta_2=\delta_2(T,\eps,R,\delta_1)>0$ such that
	$$\PP_\bz\Big\{\sup_{t\in[0,T]}|\BX_2(t)|\leq \delta_1\Big\}>1-\eps,\text{ for all }|\bx_1|\leq R, |\bx_2|\leq \delta_2.$$
\end{lm}

\begin{proof}
Because of \eqref{a.eq5}, there is an $L=L(\delta_1)$ such that whenever $U(\bx_2)>L(\delta_1)$,
$|\bx_2|<\delta_1$.
By the generalized It\^o formula, we have
\bed\label{eq-1}
\begin{aligned}
U(\BX_2(t))=&U(\bx_2)+\int_0^t\op U(\BZ(u))du+\int_0^t U_{\bx_2}\sigma_2(\BZ(u-))d\BW_2(u)\\
&+\int_0^t\int_{\Rb} [U(\BX_2(u-)+\gamma_2(\BZ(u-),\bphi))-U(\BX_2(u-))]\wN_2(du,d\bphi).
\end{aligned}
\eed
Let $H(\bz,\bphi)=U(\bx_2+\gamma_2(\bz,\bphi))-U(\bx_2)$, $\bz=(\bx_1^\top,\bx_2^\top)^\top$, and $\alpha$ be such that $0<\alpha<\frac{\alpha_0}2$, where $\alpha_0$ is in Assumption \ref{asp2}.
	From the exponential martingale inequality (see e.g., \cite[Theorem 5.2.9]{App09}), we have that with
probability greater $1-\eps$ the following inequality holds for all $t\in [0,T]$,
	$$
	\begin{aligned}
	-\int_0^t&U_{\bx_2}(\BX_2(u))\sigma_2(\BZ(u))d\mathbf W_2(u)
	-\int_0^t\int_{\Rb} H(\BZ(u),\bphi)\wN_2(du,d\bphi)\\
	\leq& \frac\alpha2 \int_0^t|U_{\bx_2}(\BX_2(u))\sigma_2(\BZ(u))|^2du
	+\frac1\alpha\int_0^t\int_{\Rb} \Big(e^{\alpha H(\BZ(u),\bphi)}-1-\alpha H(\BZ(u),\bphi)\Big)\bnu_2(d\bphi)du-\frac{\ln\eps}\alpha\\
	\leq& \frac\alpha2 \int_0^t|U_{\bx_2}(\BX_2(u))\sigma_2(\BZ(u))|^2du
	+\frac{\alpha}2\int_0^t\int_{\Rb} H^2(\BZ(u),\bphi)e^{\frac{\alpha_0}2|H(\BZ(u),\bphi)|}\bnu_2(d\bphi)du-\frac{\ln\eps}\alpha\\
	\leq& \frac\alpha2 \int_0^t|U_{\bx_2}(\BX_2(u))\sigma_2(\BZ(u))|^2du
	+\frac{C_{\alpha_0}\alpha}2\int_0^t\int_{\Rb} e^{\alpha_0|H(\BZ(u),\bphi)|}\bnu_2(d\bphi)du-\frac{\ln\eps}\alpha,
	\end{aligned}
	$$
where $C_{\alpha_0}$ is some finite constant such that $t^2e^{\frac{\alpha_0t}2}\leq C_{\alpha_0}e^{\alpha_0t},\forall t\geq 0$.
As a result, we have
\begin{equation}\label{eq-2}
\begin{aligned}\PP\bigg\{&\int_0^tU_{\bx_2}(\BX_2(u))\sigma_2(\BZ(u))d\BW_2(u)
	+\int_0^t\int_{\Rb} H(\BZ(u),\bphi)\wN_2(du,d\bphi)\\
&\geq -\frac\alpha2 \int_0^t|U_{\bx_2}(\BX_2(u))\sigma_2(\BZ(u))|^2du-\frac{C_{\alpha_0}\alpha}2\int_0^t\int_{\Rb} e^{\alpha_0|H(\BZ(u),\bphi)|}\bnu_2(d\bphi)du+\frac{\ln\eps}\alpha,\;\forall t\in [0,T]\bigg\}\\
\geq& 1-\eps.
\end{aligned}
\end{equation}
We obtain from Assumption \ref{asp2} (ii) that
$$
\begin{aligned}
-\int_0^t &[\op U](\BZ(u))du+
\frac\alpha2 \int_0^t|U_{\bx_2}(\BX_2(u))\sigma_2(\BZ(u))|^2du+\frac{C_{\alpha_0}\alpha}2\int_0^t\int_{\Rb} e^{\alpha_0|H(\BZ(u),\bphi)|}\bnu_2(d\bphi)du\\
&\leq C_{\alpha,\alpha_0}\int_0^t V_1(\BX_1(u),\0)du.
\end{aligned}
$$
In addition, by Lemma \ref{lm1},  with  probability greater than $1-\eps$,  $|\BZ(t)|\leq H_{1,T,\eps, R},\;\forall t\in [0,T]$.
As a result, with  probability greater $1-2\eps$, we have
\begin{equation}\label{eq-3}
\begin{aligned}
-\int_0^t &[\op U](\BZ(u))du+
\frac\alpha2 \int_0^t|U_{\bx_2}(\BX_2(u))\sigma_2(\BZ(u))|^2du+\frac{C_{\alpha_0}\alpha}2\int_0^t\int_{\Rb} e^{\alpha_0|H(\BZ(u),\bphi)|}\bnu_2(d\bphi)du\\
&\leq C_{T,R,\eps,\alpha,\alpha_0},
\end{aligned}
\end{equation}
for some finite constant $C_{T,R,\eps,\alpha,\alpha_0}$.
The combination of \eqref{eq-1}, \eqref{eq-2}, and \eqref{eq-3} yields that with
probability greater $1-2\eps$,
$$
U(\BX_2(t))\geq U(\bx_2)-C_{T,R,\eps,\alpha,\alpha_0}+\frac{\ln\eps}\alpha,\quad\forall t\in[0,T].
$$
Because of \eqref{a.eq5}, there exists a $\delta_2>0$ such that for all $|\bx_2|<\delta_2$,
$$
U(\bx_2)-C_{T,R,\eps,\alpha,\alpha_0}+\frac{\ln\eps}\alpha>L(\delta_1).
$$
Therefore, with   probability greater $1-2\eps$, $
\sup_{t\in[0,T]}U(\BX_2(t))\geq L(\delta_1),
$ for all $|\bx_1|<R, |\bx_2|<\delta_2$,
and thus, for all $|\bx_1|<R, |\bx_2|<\delta_2$,
$
\sup_{t\in[0,T]}|\BX_2(t)|\leq \delta_1,\;\forall t\in[0,T].
$


\end{proof}

As was mentioned, one of the main challenges in this work is the coupled interaction of $\BX_1(t)$ and $\BX_2(t)$.
To overcome this difficulty, we apply and modify the coupling method \cite{CL89}.
Let $\lambda>0$ such that $\lambda>20(1+K_2)$ with $K_2$ being the Lipschitz constant in Assumption \ref{asp1}. Consider the following coupled equations
\begin{equation}\label{e2.4}
\begin{cases}
\displaystyle d\BX_1(t)=b_1(\BX_1(t),\0)dt+\sigma_1(\BX_1(t),\0)d\BW_1(t)+\int_\Ra \gamma_1(\BX_1(t),\0, \bphi)\wN_1(dt,d\bphi),\\
\displaystyle d\wdt\BX_1(t)=b_1(\wdt\BZ(t))dt+\lambda(\BX_1(t)-\wdt \BX_1(t))dt+\sigma_1(\wdt\BZ(t))d\BW_1(t)+\int_\Ra \gamma_1(\wdt\BZ(t-), \bphi)\wN_1(dt,d\bphi), \\
\displaystyle d\wdt\BX_2(t)=b_2(\wdt\BZ(t))dt+\sigma_2(\wdt\BZ(t))d\BW_2(t)+\int_\Rb \gamma_2(\wdt\BZ(t-), \bphi)\wN_2(dt,d\bphi),\\
\BX_1(0)=\bx_1, \wdt\BX_1(0)=\wdt \bx_1, \wdt\BX_2(0)=\wdt \bx_2,
\end{cases}
\end{equation}
where $\wdt \BZ(t)=([\wdt\BX_1(t)]^\top,[\wdt\BX_2(t)]^\top)^\top$, $\wdt\bz=(\wdt\bx_1^\top,\wdt\bx_2^\top)^\top$.
We will use $\PP_{\bx_1,\wdt\bz}$ to indicate the initial conditions of the coupled equations \eqref{e2.4}.

Denote  $\Lambda_2:=\int f_2(\bx_1)\mu^*(d\bx_1)>0$,
where $\mu^*$ is the unique invariant measure of $\BX_1(t)$ when $\BX_2(t)=0$ as in Assumption \ref{asp2} (iii).
Let $C_{\alpha_0}>1$ be a constant as in Lemma \ref{lm2}, i.e., a constant satisfying $t^2e^{\frac{\alpha_0t}2}\leq C_{\alpha_0}e^{\alpha_0t},\forall t\geq 0$.
In the remaining of this section,
let
$$\gamma_0\in(0,\frac{\Lambda_1}{m_0}),\; \varsigma_0=\frac{\Lambda_1-m_0\gamma_0}3>0,\; \alpha=\frac{\varsigma_0}{2C_{\alpha_0}\Lambda_2},\;\lambda_0\in (0,\frac{\gamma_0}4),$$
and for each $\delta>0$,
$$\wdt\tau_\delta:=\inf\{t>0: |\wdt \BX_2(t)|\geq \delta e^{-\gamma_0 t}\}.$$

To proceed, we present the following lemmas. Lemma \ref{lm3} provides estimates of the coupling time of the coupled processes, while Lemma \ref{lm4} handles the diffusion and jump parts.

\begin{lm}\label{lm3}
There is a universal constant $\wdt C>1$ such that
\begin{equation}\label{l3-eq-1}
\E\Big( \sup_{t\leq \wdt\tau_\delta} e^{\lambda_0 t}|\BX_1(t)-\wdt\BX_1(t)|^2\Big)\leq \wdt C(|\bx_1-\wdt \bx_1|+\delta)^2.
\end{equation}
As a result, for any $\eps>0$ one has
 $$\PP\bigg\{\int_0^{t\wedge\wdt\tau_\delta}|v(s)|^2ds\geq \frac{\left(|\bx_1-\wdt \bx_1|+\delta\right)^2}\eps\; \text{ for some } t\geq 0\bigg\}\leq\frac{\wdt C\lambda c_\sigma\eps}{\lambda_0},$$
where $c_\sigma$ is in Assumption \ref{asp2} and
\begin{equation}\label{equationvt}
v(t):= \lambda\sigma_1^{-1}(\BX_1(t),\0)(\BX_1(t)-\wdt \BX_1(t)).
\end{equation}
\end{lm}

\begin{proof}
Use $C$ to denote a finite constant, whose values may change at difference appearances.
By the generalized It\^o
formula for jump diffusions, we have that
\begin{equation}\label{eq-main-11}
\begin{aligned}
&\!\!\! e^{\lambda_0t}\big|\BX_1(t)-\wdt\BX_1(t)\big|^2\\
& =\left|\bx_1-\wdt\bx_1\right|^2+\int_0^t e^{\lambda_0s}(\lambda_0-\lambda) \big|\BX_1(s)-\wdt\BX_1(t)\big|^2ds\\
&\ +2\int_0^t e^{\lambda_0s}\big(\BX_1(s)-\wdt\BX_1(s)\big)^\top\big(b_1(\BX_1(s),\0)-b_1(\wdt\BX_1(s),\wdt\BX_2(s))\big)ds\\
&\ +\int_0^t e^{\lambda_0s}\tr\Big[\big(\sigma_1(\BX_1(s),\0)-\sigma_1(\wdt\BX_1(s),\wdt\BX_2(s))\big)\big(\sigma_1(\BX_1(s),\0)-\sigma_1(\wdt\BX_1(s),\wdt\BX_2(s))\big)^\top\Big]ds\\
&\ +\int_0^t \! e^{\lambda_0s}\!\int_\Ra \! \Big[\big|\BX_1(s)-\wdt\BX_1(s)+\gamma_1(\BX_1(s-),\0,\bphi)
-\gamma_1(\wdt\BX_1(s),\wdt\BX_2(s),\bphi)\big|^2\!
-\!\big|\BX_1(s)-\wdt\BX_1(s)\big|^2\\
&\qquad\qquad-2\big(\BX_1(s)-\wdt\BX_1(s)\big)^\top \big(\gamma_1(\BX_1(s),\0,\bphi)-
\gamma_1(\wdt\BX_1(s),\wdt\BX_2(s),\bphi)\big)\Big]\bnu_1(d\bphi)ds\\
&+2\int_0^t e^{\lambda_0s}\big(\BX_1(s-)-\wdt\BX_1(s-)\big)^\top\big(\sigma_1(\BX_1(s-),\0)-\sigma_1(\wdt\BX_1(s-),\wdt\BX_2(s-))\big)d\BW_1(s)\\
&\ +\int_0^t e^{\lambda_0s}\int_\Ra\Big[ \big|\BX_1(s-)-\wdt\BX_1(s-)+\gamma_1(\BX_1(s-),\0,\bphi)-\gamma_1(\wdt\BX_1(s-),\wdt\BX_2(s-),\bphi)\big|^2\\
&\qquad\qquad -\big|\BX_1(s-)-\wdt\BX_1(s-)\big|^2\Big]\wN_1(ds,d\bphi).
\end{aligned}
\end{equation}
By
virtue of Kunita's first inequality
 \cite[Theorem 4.4.23, p. 265]{App09},
for all $T\geq 0$,
\begin{equation}\label{eq-supN}
\begin{aligned}
\E &\sup_{t\in [0,T]}\Big|\int_0^{t\wedge \wdt\tau_\delta} \int_\Ra e^{\lambda_0s}\Big[\big|\BX_1(s-)-\wdt\BX_1(s-)+\gamma_1(\BX_1(s-),\0,\bphi)-\gamma_1(\wdt\BX_1(s-),\wdt\BX_2(s-),\bphi)\big|^2\\
&\qquad\qquad-\big|\BX_1(s-)-\wdt\BX_1(s-)\big|^2\Big]\wN_1(ds,d\bphi)\Big|^2\\
&\leq C\E \int_0^{T\wedge \wdt\tau_\delta}\int_\Ra e^{2\lambda_0s}\Big[\big|\BX_1(s-)-\wdt\BX_1(s-)+\gamma_1(\BX_1(s-),\0,\bphi)-\gamma_1(\wdt\BX_1(s-),\wdt\BX_2(s-),\bphi)\big|^2\\
&\qquad\qquad-\big|\BX_1(s-)-\wdt\BX_1(s-)\big|^2\Big]^2\bnu_1(d\bphi)ds\\
&\leq C\E \Big(\int_0^{T\wedge \wdt\tau_\delta}e^{2\lambda_0s}\big|\BX_1(s)-\wdt\BX_1(s)\big|^4ds +\delta^4 \int_0^{T\wedge\wdt\tau_\delta} e^{(-4\gamma_0+2\lambda_0) s}ds\Big)\\
&\leq C\Big(\delta^4+\E\int_0^{T\wedge\wdt\tau_\delta}e^{2\lambda_0s}\big|\BX_1(s)-\wdt\BX_1(s)\big|^4ds\Big).
\end{aligned}
\end{equation}
On the other hand,
by the Burkholder-Davis-Gundy inequality  \cite[Theorem 2.13, p. 70]{MY06}, one has
\begin{equation}\label{eq-supW}
\begin{aligned}
\E&\sup_{t\in [0,T]}\Big|\int_0^{t\wedge\wdt\tau_\delta}e^{\lambda_0s}\big(\BX_1(s-)-\wdt\BX_1(s-)\big)^\top\big(\sigma_1(\BX_1(s-),\0)-\sigma_1(\wdt\BX_1(s-),\wdt\BX_2(s-))\big)d\BW_1(s)\Big|^2\\
&\leq C\E\Big(\int_0^{T\wedge \wdt\tau_\delta}e^{2\lambda_0s}\big|\BX_1(s)-\wdt\BX_1(s)\big|^4ds+\delta^4\int_0^{T\wedge\wdt\tau_\delta} e^{(-4\gamma_0+2\lambda_0)s}ds\Big)\\
&\leq C\Big(\delta^4+\E\int_0^{T\wedge\wdt\tau_\delta}e^{2\lambda_0s}\big|\BX_1(s)-\wdt\BX_1(s)\big|^4ds\Big).
\end{aligned}
\end{equation}
Now, applying \eqref{eq-supN} and \eqref{eq-supW} to \eqref{eq-main-11}, and using the Lipschitz continuity of $b_1(\cdot,\cdot)$, $\sigma_1(\cdot,\cdot)$, and $\gamma_1(\cdot,\cdot,\cdot)$, we obtain that
\begin{equation}\label{eq-supx-x}
\begin{aligned}
\E\sup_{t\leq T\wedge \wdt\tau_\delta}& e^{\lambda_0 t}\big|\BX_1(t)-\wdt\BX_1(t)\big|^2\\
&\leq \left|\bx_1-\wdt\bx_1\right|^2+
C\Big(\delta^2+\int_0^T e^{(2\lambda_0-\gamma_0)s}ds+\Big[\E\int_0^{T\wedge\wdt\tau_\delta}e^{2\lambda_0s}\big|\BX_1(s)-\wdt\BX_1(s)\big|^4ds\Big]^{\frac 12}\Big).
\end{aligned}
\end{equation}

To proceed, we
estimate $\E \int_0^{t \wedge \wdt\tau_\delta}e^{2\lambda_0s}\big|\BX_1(s)-\wdt\BX_1(s)\big|^4ds$.
Using the generalized It\^o formula again, we have
\begin{equation}\label{eq-mu4}
\begin{aligned}
&\!\!\!\! e^{2\lambda_0t}\big|\BX_1(t)-\wdt\BX_1(t)\big|^4\\
& =\left|\bx_1-\wdt\bx_1\right|^4+\int_0^t e^{2\lambda_0s}(2\lambda_0-\lambda) \big|\BX_1(s)-\wdt\BX_1(t)\big|^4ds\\
&\ +\int_0^t 4e^{2\lambda_0s}\big|\BX_1(s)-\wdt\BX_1(s)\big|^2\big(\BX(s)-\wdt\BX_1(s)\big)^\top\big(b_1(\BX_1(s),\0)-b_1(\wdt\BX_1(s),\wdt\BX_2(s))\big)ds\\
&\ +3\int_0^t e^{2\lambda_0s}\big|\BX_1(s)-\wdt\BX_1(s)\big|^2\\
&\qquad\times \tr\Big[\big(\sigma_1(\BX_1(s),\0,\bphi)-\sigma_1(\wdt\BX_1(s),\wdt\BX_2(s),\bphi)\big)\big(\sigma_1(\BX_1(s),\0,\bphi)-\sigma_1(\wdt\BX_1(s),\wdt\BX_2(s),\bphi)\big)^\top\Big]ds\\
&\ +\int_0^t \int_\Ra e^{2\lambda_0s}\Big[\big|\BX_1(s)-\wdt\BX_1(s)+\gamma_1(\BX_1(s-),\0,\bphi)-\gamma_1(\wdt\BX_1(s),\wdt\BX_2(s),\bphi)\big|^4-\big|\BX_1(s)-\wdt\BX_1(s)\big|^4\\
&\qquad-4\big|\BX_1(s)-\wdt\BX_1(s)\big|^2\big(\BX_1(s)-\wdt\BX_1(s)\big)^\top \big(\gamma_1(\BX_1(s),\0,\bphi)-\gamma_1(\wdt\BX_1(s),\wdt\BX_2(s),\bphi)\big)\Big]\bnu_1(d\bphi)ds\\
&\ +4\int_0^t e^{2\lambda_0s}\big|\BX_1(s-)-\wdt\BX_1(s-)\big|^2\\
&\qquad\times \big(\BX_1(s-)-\wdt\BX_1(s-)\big)^\top\big(\sigma_1(\BX_1(s-),\0)-\sigma_1(\wdt\BX_1(s-),\wdt\BX_2(s-))\big)d\BW_1(s)\\
&\ +\int_0^t \int_\Ra e^{2\lambda_0s}\Big[\big|\BX_1(s-)-\wdt\BX_1(s-)+\gamma_1(\BX_1(s-),\0,\bphi)-\gamma_1(\wdt\BX_1(s-),\wdt\BX_2(s-),\bphi)\big|^4\\
&\qquad\qquad-\big|\BX_1(s-)-\wdt\BX_1(s-)\big|^4\Big]\wN_1(ds,d\bphi).
\end{aligned}
\end{equation}
Taking the expectation on both sides of \eqref{eq-mu4}, using the Lipschitz continuity of $b_1(\cdot,\cdot)$, $\sigma_1(\cdot,\cdot)$, $\gamma_1(\cdot,\cdot,\cdot)$, and noting
$\lambda$ being chosen to be sufficiently large,
we obtain
\begin{equation}\label{u-eq3}
\begin{aligned}
d \Big[&\E e^{2\lambda_0(t\wedge\wdt\tau_\delta)}\big|\BX_1(t\wedge\wdt\tau_\delta)-\wdt\BX_1(t\wedge\wdt\tau_\delta)\big|^4\Big]\\
&\leq \E\Big[-D_1 e^{2\lambda_0(t\wedge\wdt\tau_\delta)} \big|\BX_1(t)-\wdt\BX_1(t)\big|^4+D_2\delta^4e^{(-4\delta_0+2\lambda_0)(t\wedge\wdt\tau_\delta)}\Big]dt,
\end{aligned}
\end{equation}
for some finite positive constants $D_1$ and $D_2$.
Eq. \eqref{u-eq3} implies that
\begin{equation*}
\E\int_0^{t\wedge\wdt\tau_\delta} e^{2\lambda_0s}\big|\BX_1(s)-\wdt\BX_1(s)\big|^4 ds\leq C\left(|\bx_1-\wdt\bx_1|^4+\delta^4\right)\text{ for all }t\geq 0,\end{equation*}
and thus,
\begin{equation}\label{u-eq5}
\Big[\E\int_0^{t\wedge\wdt\tau_\delta} e^{2\lambda_0s}\big|\BX_1(s)-\wdt\BX_1(s)\big|^4 ds\Big]^{\frac 12}\leq C\left(|\bx_1-\wdt\bx_1|^2+\delta^2\right)\text{ for all }t\geq 0.\end{equation}
Combining \eqref{u-eq5} and \eqref{eq-supx-x}, we get that
\begin{equation}
\E\sup_{t\leq T\wedge \wdt\tau_\delta} e^{\lambda_0 t}\big|\BX_1(t)-\wdt\BX_1(t)\big|^2\leq C\big(|\bx_1-\wdt\bx_1|^2+\delta^2\big)\text{ for all }T\geq 0.
\end{equation}
Therefore, \eqref{l3-eq-1} is proved.

Now, we consider
the second part.
By virtue of the definition of $v(t)$ in \eqref{equationvt} and Assumption \ref{asp2} (iv),
\begin{equation}\label{l3-eq-2}
\begin{aligned}
\PP&\bigg\{\int_0^{t\wedge\wdt\tau_\delta}|v(s)|^2ds\geq \frac{\left(|\bx_1-\wdt \bx_1|+\delta\right)^2}\eps \;\text{ for some } t\geq 0\bigg\}\\
&\leq \PP\bigg\{\int_0^{\wdt\tau_\delta}\big|\lambda\sigma_1^{-1}(\BX_1(s),\0)(\BX_1(s)-\wdt \BX_1(s))\big|^2ds\geq \frac{\left(|\bx_1-\wdt \bx_1|+\delta\right)^2}\eps\bigg\}\\
&\leq \PP\bigg\{\int_0^{\wdt\tau_\delta}|\BX_1(s)-\wdt \BX_1(s)|^2ds\geq \frac{\left(|\bx_1-\wdt \bx_1|+\delta\right)^2}{\lambda c_\sigma\eps}\bigg\}. 
\end{aligned}
\end{equation}
A standard calculation shows that for the integrable function $h(s)$,
$$
\int_0^t h(s)ds= \int_0^t \frac{e^{\lambda_0 s} h(s)}{e^{\lambda_0s}}ds\leq \sup_{s\in[0,t]}e^{\lambda_0s}h(s)\int_0^t e^{-\lambda_0s}ds \leq \frac 1{\lambda_0}\sup_{s\in[0,t]}e^{\lambda_0s}h(s).
$$
Therefore, it follows from \eqref{l3-eq-2} that
\begin{equation}\label{l3-eq-2-1}
\begin{aligned}
\PP&\bigg\{\int_0^{t\wedge\wdt\tau_\delta}|v(s)|^2ds\geq \frac{\left(|\bx_1-\wdt \bx_1|+\delta\right)^2}\eps\; \text{ for some } t\geq 0\bigg\}\\
&\leq \PP\bigg\{\sup_{t\leq \wdt\tau_\delta}e^{\lambda_0t}|\BX_1(t)-\wdt \BX_1(t)|^2\geq \frac{\lambda_0\left(|\bx_1-\wdt \bx_1|+\delta\right)^2}{\lambda c_\sigma\eps}\bigg\}\\
&\leq \dfrac{ \lambda c_\sigma \eps}{\lambda_0(|\bx_1-\wdt\bx_1|+\delta)^2}\E \sup_{t\leq \wdt\tau_\delta}e^{\lambda_0t}|\BX_1(t)-\wdt\BX_1(t)|^2\\
&\leq \frac{\wdt C\lambda c_\sigma\eps}{\lambda_0} \text{ (due to \eqref{l3-eq-1})}.
\end{aligned}
\end{equation}
As a result, the proof is complete.
\end{proof}

\begin{lm}\label{lm4}
	Let $\wdt C$ be as in Lemma \ref{lm1} and $L$ be a Lipschiz constant of $f_1(\cdot)$ and $f_2(\cdot)$.
Suppose $|\bx_1-\wdt \bx_1|+\delta<1$.
	Then we have
$$
\begin{aligned}
\PP_{\bx_1,\wdt\bz}\bigg\{\int_0^t&U_{\bx_2}(\wdt\BX_2(u))\sigma_2(\wdt\BZ(u))d\BW_2(u)
+\int_0^t\int_{\Rb} H(\wdt\BZ(u),\bphi)\wN_2(du,d\bphi)\\
&\geq -C_{\alpha_0}\alpha \int_0^t f_2(\BX_1(u))du+\frac{\ln\eps}\alpha
-\frac{2L\wdt C}{\eps\lambda_0}, \;\forall 0\leq t\leq \wdt\tau_\delta\bigg\}\geq 1-2\eps.
\end{aligned}
$$
	\end{lm}

\begin{proof}
	Let $H(\bz,\bphi)=U(\bx_2+\gamma_2(\bz,\bphi))-U(\bx_2)$.
	From the exponential martingale inequality (see e.g., \cite[Theorem 5.2.9]{App09}), we have with
probability greater than
$1-\eps$,
	$$
	\begin{aligned}
	-\int_0^t&U_{\bx_2}(\wdt\BX_2(u))\sigma_2(\wdt\BZ(u))d\BW_2(u)
	-\int_0^t\int_{\Rb} H(\wdt\BZ(u),\bphi)\wN_2(du,d\bphi)\\
	\leq& \frac\alpha2 \int_0^t|U_{\bx_2}(\wdt\BX_2(u))\sigma_2(\wdt\BZ(u))|^2du
	+\frac1\alpha\int_0^t\int_{\Rb} \left(e^{\alpha H(\wdt\BZ(u),\bphi)}-1-\alpha H(\wdt\BZ(u),\bphi)\right)\bnu_2(d\bphi)du-\frac{\ln\eps}\alpha\\
	\leq& \frac\alpha2 \int_0^t|U_{\bx_2}(\wdt\BX_2(u))\sigma_2(\wdt\BZ(u))|^2du
	+\frac{\alpha}2\int_0^t\int_{\Rb} H^2(\wdt\BZ(u),\bphi)e^{\frac{\alpha_0}2|H(\wdt\BZ(u),\bphi)|}\bnu_2(d\bphi)du-\frac{\ln\eps}\alpha\\
	\leq& \frac\alpha2 \int_0^t|U_{\bx_2}(\wdt\BX_2(u))\sigma_2(\wdt\BZ(u))|^2du
	+\frac{C_{\alpha_0}\alpha}2\int_0^t\int_{\Rb} e^{\alpha_0|H(\wdt\BZ(u),\bphi)|}\bnu_2(d\bphi)du-\frac{\ln\eps}\alpha.
	\end{aligned}
	$$
As a result, $\PP(\Omega_1)\geq 1-\eps$, where
$$
\begin{aligned}\Omega_1:=\bigg\{\int_0^t&U_{\bx_2}(\wdt\BX_2(u))\sigma_2(\wdt\BZ(u))d\BW_2(u)
	+\int_0^t\int_{\Rb} H(\wdt\BZ(u),\bphi)\wN_2(du,d\bphi)\\
&\leq \frac\alpha2 \int_0^t|U_{\bx_2}(\wdt\BX_2(u))\sigma_2(\wdt\BZ(u))|^2du+\frac{C_{\alpha_0}\alpha}2\int_0^t\int_{\Rb} e^{\alpha_0|H(\wdt\BZ(u),\bphi)|}\bnu_2(d\bphi)du-\frac{\ln\eps}\alpha\bigg\}.
\end{aligned}
$$
On the other hand, by \eqref{l3-eq-1}, one has $\PP(\Omega_2)\geq 1-\eps$, where
$$
\Omega_2:=\bigg\{\sup_{t\leq\wdt\tau_\delta}e^{\lambda_0t}|\BX_1(t)-\wdt\BX_1(t)|^2\geq\frac{\wdt C}\eps\bigg\}.
$$
For $t\leq\wdt\tau_\delta$ and
$\omega\in\Omega_1\cap\Omega_2$,
	 we have
	\begin{align*}
	-\int_0^t&U_{\bx_2}(\wdt\BX_2(u))\sigma_2(\wdt\BZ(u))d\BW_2(u)
	-\int_0^t\int_{\Rb} H(\wdt\BZ(u),\bphi)\wN_2(du,d\bphi)\\
	&\leq \frac\alpha2 \int_0^t|U_{\bx_2}(\wdt\BX_2(u))\sigma_2(\wdt\BZ(u))|^2du
	+\frac{C_{\alpha_0}\alpha}2\int_0^t\int_{\Rb} e^{\alpha_0|H(\wdt\BZ(u),\bphi)|}\bnu_2(d\bphi)du-\frac{\ln\eps}\alpha\\
	&\leq C_{\alpha_0}\alpha \int_0^t f_2(\wdt\BX_1(u))du-\frac{\ln\eps}\alpha\\
	&\leq C_{\alpha_0}\alpha \int_0^t f_2(\BX_1(u))du-\frac{\ln\eps}\alpha
	+L\int_0^t |\BX_1(u)-\wdt\BX_1(u)|du\\
	&\leq C_{\alpha_0}\alpha \int_0^t f_2(\BX_1(u))du-\frac{\ln\eps}\alpha
	+\frac{L\sqrt{\wdt C}}{\sqrt\eps}\int_0^t  e^{-\frac{\lambda_0u}2 } du\\
	&\leq C_{\alpha_0}\alpha \int_0^t f_2(\BX_1(u))du-\frac{\ln\eps}\alpha
	+\frac{2L\wdt C}{\eps\lambda_0}.
	\end{align*}
Therefore, the proof is complete.
\end{proof}

\begin{proof}[Proof of Theorem \ref{thm-main}]
	We can assume that $\eps\in(0,\frac12)$ and $e^{-3\eps}\geq 1-4\eps$.
	Let
	\begin{equation}\label{eq-delta}
	\delta\leq \frac{\eps^3}{-\ln\eps}\wedge \frac{\eps^2}{2}\wedge \frac{\lambda_0}{2\wdt C\lambda c_\sigma},
	\end{equation}
	where $\wdt C$ is in Lemma \ref{lm1}.
We have from the definitions of $\Lambda_1,\Lambda_2$ that
$$	\PP_{\bx_1,\0}\left\{\limsup_{t\to\infty}\frac{1}t\int_0^tf_1(\BX_1(u))du= \Lambda_1\right\}=\PP_{\bx_1,\0}\left\{\limsup_{t\to\infty}\frac{1}t\int_0^tf_2(\BX_1(u))du= \Lambda_2\right\}=1.$$
It is noted that $m_0\gamma_0+2\varsigma_0<\Lambda_1$ so that there exists $\txe\geq\frac1{\varsigma_0}\left(\frac{-\ln\eps}\alpha
+\frac{4L\wdt C}{\eps\lambda_0}\right)$ such that $\PP_{\bx_1,\0}(\Omega_3)\geq 1-\eps$, where
$$\Omega_3:=\left\{\frac{1}t\int_0^tf_1(\BX_1(u))du\geq m_0\gamma_0+2\varsigma_0\text{ and } \frac{1}t\int_0^tf_2(\BX_1(u))du\leq\Lambda_2, \, t\geq \txe\right\}.$$
Let $\mxe>0$ be such that $\PP_{\bx_1,\0}(\Omega_4)\geq 1-\eps$, where
$$\Omega_4:=\left\{\frac{1}t\int_0^tf_1(\BX_1(u))du- \frac{C_{\alpha_0}\alpha}t\int_0^tf_2(\BX_1(u))du+\frac{\ln\eps}\alpha
-\frac{4L\wdt C}{\eps\lambda_0}\geq -\mxe, \, t\leq \txe\right\},$$
and $\Omega_1$, $\Omega_2$ be as in Lemma \ref{lm4}.
By the generalized It\^o formula, we have
\bed\label{eq-1-1}
\begin{aligned}
	U(\wdt\BX_2(t))=&U(\bx_2)+\int_0^t\op U(\wdt\BZ(u))du+\int_0^t U_{\bx_2}\sigma_2(\wdt\BZ(u-))d\BW_2(u)\\
	&+\int_0^t\int_{\Rb} [U(\wdt\BX_2(u-)+\gamma_2(\wdt\BZ(u-),\bphi))-U(\wdt\BX_2(u-))]\wN_2(du,d\bphi).
\end{aligned}
\eed
Definitions of $\Omega_1$, $\Omega_2$, and $\Omega_4$ lead to
\bea
 U(\wdt\BX_2(t))\geq U(\wdt\bx_2) -\mxe,\text{ for all } t\leq \txe\wedge\wdt\tau_\delta,\,\omega\in\Omega_1\cap\Omega_2\cap\Omega_4.
 \eea
Let $\deltaxe\in(0,\frac\delta2)$ such that $U(\bx_2)-U(\bx_2')>\mxe$
if $|\bx_2|\leq \deltaxe, |\bx_2'|\geq\delta e^{-\txe}$.
Such a $\deltaxe$ exists owing to \eqref{a.eq5}.
For $|\wdt\bx_2|<\deltaxe$ and $\omega\in\Omega_2\cap\Omega_3$,
we must have
 $\wdt\tau_\delta>\txe$. Otherwise,
if $\wdt\tau_\delta\leq \txe$, we have
 $U(\wdt\BX_2(\wdt\tau_\delta))\geq U(\wdt\bx_2) -\mxe,$
 which implies that
 $|\wdt\BX_2(\wdt\tau_\delta)|<\delta e^{-\txe}$,
 and again contradict to the definition of $\wdt\tau_\delta$.

For $\omega\in \cap_{i=1}^4\Omega_i$, we have $\wdt\tau_\delta>\txe$. From \eqref{eq-1-1}, Assumption \ref{asp2}, and Lemma \ref{lm4}, one has
\begin{equation}\label{eq-9-2}
\begin{aligned}
U(\wdt\BX_2(t))&\geq U(\wdt\bx_2)+\int_0^tf_1(\wdt\BX_1(u))du- C_{\alpha_0}\alpha\int_0^tf_2(\BX_1(u))du+\frac{\ln\eps}\alpha
-\frac{2L\wdt C}{\eps\lambda_0} \\
&\geq U(\wdt\bx_2)+\int_0^tf_1(\BX_1(u))du-L\int_0^t |\BX_1(u)-\wdt\BX_1(u)|du\\
&\quad- C_{\alpha_0}\alpha\int_0^tf_2(\BX_1(u))du+\frac{\ln\eps}\alpha
-\frac{2L\wdt C}{\eps\lambda_0} \\
&\geq U(\wdt\bx_2)+(m_0\gamma_0+2\varsigma_0)t-\frac{L\wdt C}{\eps}\int_0^t  e^{-\frac{\lambda_0u}2 } du-C_{\alpha_0}\alpha\Lambda_2 t +\frac{\ln\eps}\alpha
-\frac{2L\wdt C}{\eps\lambda_0}\\
&\geq U(\wdt\bx_2)+m_0\gamma_0t+\varsigma_0\txe+\frac{\ln\eps}\alpha
-\frac{4L\wdt C}{\eps\lambda_0}\\
&\geq U(\wdt\bx_2)+ m_0\gamma_0 t, \text{ for all }\omega\in\cap_{i=1}^4\Omega_i, t\in[\txe,\wdt\tau_\delta).
\end{aligned}
\end{equation}
The combination of \eqref{eq-9-2} and \eqref{a.eq5} implies that
$
\ln |\wdt\BX_2(t)|\leq \ln|\wdt\bx_2|-\gamma_0 t, \forall t\in[T,\wdt\tau_\delta),
$
so that $|\wdt\BX_2(t)|<\frac{\delta}2 e^{-\gamma_0 t}, t\in[T,\wdt\tau_\delta).
$
As a result, if $\wdt\bx_2<\deltaxe$, $\wdt\tau_\delta=\infty$ for all $\omega\in \cap_{i=1}^4\Omega_i$.
It yields that
for all $\omega\in \cap_{i=1}^4\Omega_i$,
$$
\liminf_{t\to\infty}\frac{U(\wdt\BX_2(t))}{t}\geq m_0\gamma_0\text{ and }\limsup_{t\to\infty}\frac{\ln|\wdt\BX_2(t)|}{t}\leq -\gamma_0.
$$
An application of Lemma \ref{lm3}
leads to
$\PP_{\bx_1,\wdt\bz}(\Omega_5)\geq 1-\eps$, where
 $$\Omega_5:=\bigg\{\int_0^{t\wedge\wdt\tau_\delta}|v(s)|^2ds\leq \frac{\wdt C\lambda c_\sigma(|\bx_1-\wdt \bx_1|+\delta)^2}{\lambda_0\eps}\; \forall t\geq 0\bigg\}.$$
 If $|\bx_1-\wdt \bx_1|\leq \delta$, we have from \eqref{eq-delta} that for all $\omega\in\Omega_5$
 $$
 \int_0^{t\wedge\wdt\tau_\delta}|v(s)|^2ds\leq \frac{4\wdt C\lambda c_\sigma\delta^2}{\lambda_0\eps}\leq \frac{2\delta}\eps\leq\eps\; \forall t\geq 0.
 $$
By the exponential martingale inequality \cite[Theorem 7.4, page 44]{Mao97},
we have $\PP_{\bx_1,\wdt\bz}(\Omega_6)\geq 1-e^{\frac{\eps^3}{\delta}}=1-\eps$ because $\delta\leq \frac{\eps^3}{-\ln\eps}$, where
$$
\Omega_6:=\left\{-\int_0^{t}v(s) d\mathbf W_1(s)\leq \frac{\eps^2}{2\delta}\int_0^{t}|v(s)|^2ds+\eps, t\leq \wdt\tau_\delta\right\}.
$$
Let $\hat\BZ(t)=([\hat \BX_1(t)]^\top,[\hat \BX_2(t)]^\top)^\top$ be solution of following coupled system
\begin{equation}\label{ehatZ}
\begin{cases}
\displaystyle d\BX_1(t)=b_1(\BX_1(t),\0)dt+\sigma_1(\BX_1(t),\0)d\BW_1(t)+\int_\Ra \gamma_1(\BX_1(t),\0, \bphi)\wN_1(dt,d\bphi),\\
\displaystyle d\hat\BX_1(t)=b_1(\hat\BZ(t))dt+\1_{\{t\leq\wdt\tau_\delta\}}\lambda(\BX_1(t)-\hat \BX_1(t))dt+\sigma_1(\hat\BZ(t))d\BW_1(t)+\int_\Ra \gamma_1(\hat\BZ(t-), \bphi)\wN_1(dt,d\bphi), \\
\displaystyle d\hat\BX_2(t)=b_2(\hat\BZ(t))dt+\sigma_2(\hat\BZ(t))d\BW_2(t)+\int_\Rb \gamma_2(\hat\BZ(t-), \bphi)\wN_2(dt,d\bphi),\\
\BX_1(0)=\bx_1, \hat\BX_1(0)=\wdt \bx_1, \wdt\BX_2(0)=\wdt \bx_2,
\end{cases}
\end{equation}
The Cameron-Martin-Girsanov theorem implies that there exist a measure $\Q_{\bx_1,\wdt\bz}$ such that under $\Q_{\bx_1,\wdt\bz}$, $\mathbf W_1(t)+\int_0^{t\wedge\wdt\tau_\delta}v(s)ds$ is a Wiener process so that $\hat\BZ(t)=([\hat \BX_1(t)]^\top,[\hat \BX_2(t)]^\top)^\top$ under $\Q_{\bx_1,\wdt\bz}$ is the solution to \eqref{e2.3} with initial value $\wdt\bz$.
Because $\hat \BZ(t)=\wdt \BZ(t)$ for all $t$ if $\wdt\tau_\delta=\infty$,
 $$\liminf_{t\to\infty}\frac{U(\hat\BX_2(t))}t=\liminf_{t\to\infty}\frac{U(\wdt\BX_2(t))}t\geq m_0\gamma_0>0$$ in the set $\cap_{i=1}^4\Omega_i$.

On the other hand, if $|\bx_1-\wdt\bx_1|+|\wdt\bx_2|\leq\deltaxe$, for $\omega\in\Omega_5\cap\Omega_6$, we have
$$\dfrac{d\Q_{\bx_1,\wdt\bz}}{d\PP_{\bx_1,\wdt\bz}}=\exp\left\{\int_0^{\wdt\tau_\delta}v(s) d\BW_1(s)-\int_0^{\wdt\tau_\delta}|v(s)|^2ds\right\}\geq e^{-\eps-\eps-\eps}\geq 1-4\eps.$$
Since
$$\PP_{\bx_1,\wdt\bz}(\cap_{i=1}^6\Omega_i)\geq 1-6\eps,$$
we have
$$\Q_{\bx_1,\wdt\bz}(\cap_{i=1}^6\Omega_i)\geq (1-6\eps)(1-4\eps)=1-24\eps,$$
which implies
$$\Q_{\bx_1,\wdt\bz}\left\{\liminf_{t\to\infty}\frac{U(\hat\BX_2(t))}t\geq m_0\gamma_0\right\}\geq 1-24\eps.$$
Therefore, we obtain that
if $|\bx_1-\wdt\bx_1|+|\wdt\bx_2|\leq\deltaxe$ then
$$\PP_{\bx_1,\wdt\bz}\Big\{\liminf_{t\to\infty}\frac{U(\BX_2(t))}t\geq m_0\gamma_0\Big\}\geq 1-24\eps.$$
Hence, scaling $\eps$ by $\frac{\eps}{24}$,
we 
obtain \eqref{eq-thm-1},
which together with \eqref{a.eq5} implies \eqref{eq-thm-2}.
The proof is complete.
\end{proof}

\section{Stability of Linearizable Systems and Fast-Slow Systems}\label{sec:app1}
In this section, we consider the system of equations
\begin{equation}\label{e4.1}
\begin{cases}
\displaystyle
d\BY_1(t)=b_1(\BY_1(t),\BY_2(t))dt+\sigma_1(\BY_1(t),\BY_2(t))d\BW_1(t)+\int_\Ra \gamma_1(\BY_1(t-),\BY_2(t-), \bphi)\wN_1(dt,d\bphi), \\
\displaystyle d\BY_2(t)=b_2(\BY_1(t),\BY_2(t))dt+\sigma_2(\BY_1(t),\BY_2(t))d\BW_2(t)+\int_\Ra \gamma_2(\BY_1(t-),\BY_2(t-), \bphi) \wN_2(dt,d\bphi).\end{cases}
\end{equation}
Under the condition
that the second equation can be linearized (see Assumption \ref{asp3.1} below),
we examine the stability of $\BY_2(\cdot)$.
Then we consider the case when the two components have different time scales.
\subsection{Stability of Linearizable Systems}

\begin{asp}\label{asp3.1}
{\rm	Assume that
	the assumptions in Theorem \ref{thm-main} for $b_1,\sigma_1,\gamma_1$ still hold
and that $b_2,\sigma_2,\gamma_2$ are linearizable in $\by_2$. That is, there exist matrices $B_2(\by_1),\Sigma_{21}(\by_1), \dots,\Sigma_{2d_2}(\by_1),\Gamma_2(\by_1)$ bounded in $\by_1$ such that $\Sigma_{2i}(\by_1)$ has bounded right inverse and
$$
\begin{aligned}
&\norm{b_2(\bz)-B_2(\by_1)\by_2}\leq  o(\by_2)V(\by_1),\\
&\|\sigma_2(\bz)-[\Sigma_{21}(\by_1)\by_2,\dots,\Sigma_{2d_2}(\by_1)\by_2]\|\leq   o(\by_2)\sqrt{V(\by_1)},\\
&\gamma_2(\bz,\bphi)=\Gamma_2(\by_1,\bphi)\by_2+o(\by_2)V(\by_1),\\
&\|b_2(\bz)\|+\|\sigma_2(\bz)\sigma_2^\top(\bz)\|\leq K V(\by_1),
\end{aligned}
$$
where $o(\by_2)$ is a matrix or vector depending on $\bz$ satisfying
$\lim_{\by_2\to0 } \frac{\sup_{\by_1\in\RX}\{ |o(\by_2)|\}}{|\by_2|}=0.
$}
\end{asp}

Let $\Theta(t)=\frac{\BY_2(t)}{|\BY_2(t)|}$, $R(t)=|\BY_2(t)|^2$, by the generalized It\^o
formula for jump diffusions
we have the following equations for $\BY_1(t),\Theta(t),R(t)$
\begin{equation}\label{e4.3-1}
{\small
\begin{cases}
\displaystyle
d\BY_1(t)=b_1(\BZ(t))dt+\sigma_1(\BZ(t))d\BW_1(t)+\int_\Ra \gamma_1(\BZ(t-), \bphi)\wN_1(dt,d\bphi), \\
\displaystyle d\Theta(t)=g_1(\BY_1(t),\Theta(t),  R(t))dt+g_2(\BY_1(t),\Theta(t),  R(t))d\BW_2(t)+\int_{\Rb}g_3(\BY_1(t),\Theta(t),  R(t),\bphi)\wN_2(dt,d\bphi),\\
\displaystyle d R(t)=h_1(\BY_1(t),\Theta(t),  R(t))dt+h_2(\BY_1(t),\Theta(t),  R(t))d\BW_2(t)+\int_{\Rb}h_3(\BY_1(t),\Theta(t),  R(t),\bphi)\wN_2(dt,d\bphi),
\end{cases}}
\end{equation}
where $\BZ(t)=([\BY_1(t)]^\top,[\BY_2(t)]^\top)^\top=([\BY_1(t)]^\top,\sqrt{R(t)}[\Theta(t)]^\top)^\top$.
In \eqref{e4.3-1}, $g_i$  and $h_i$ with $i=1,2,3$ are given as follows. If we denote $\bz=(\by_1^\top,\by_2^\top)^\top=(\by_1^\top,\sqrt{r}\theta^\top)^\top$ then
\bea
g_1(\by_1,\theta,r)\ad= \frac{b_2(\bz)}{|\by_2|}
-\frac{\sigma_2(\bz)\sigma_2^\top(\bz)\by_2}{|\by_2|^3}+\Big(-\by_2^\top b_2(\bz)
-\frac12\trace(\sigma_2^\top(\bz)
\sigma_2(\bz))+\frac{3|\by_2^\top\sigma_2(\bz)|^2}
{2|\by_2|^2}\Big]\Big)\frac{\by_2}{|\by_2|^3}\\
\aad\ +\int_{\Rb}\left(\frac{\by_2+\gamma_2(\bz,\bphi)}{|\by_2+\gamma_2(\bz,\bphi)|}-\frac{\by_2}{|\by_2|}-\frac{|\by_2|^2\gamma_2(\bz,\bphi)-(\by_2^\top \gamma_2(\bz,\bphi))\by_2}{|\by_2|^3}\right)\bnu_2(d\bphi),\\
\ad= B_2(\by_1)\theta -\sum_{l=1}^{d_2}\big[\theta^\top\Sigma_{2l}(\by_1)\theta\big]\left[\Sigma_{2l}(\by_1)\theta\right]\\
\aad\ +\Big(-\theta^\top B_2(\by_1)\theta +\frac12\sum_{l=1}^{d_2}\big[-|\Sigma_{2l}(\by_1)\theta|^2+3|\theta^\top\Sigma_{2l}(\by_1)\theta|^2\big]\Big)\theta\\
\aad\ +\int_{\Rb}\Big(\frac{\theta+\Gamma_2(\by_1,\bphi)\theta}{|\theta+\Gamma_2(\by_1,\bphi)\theta|}-\theta-|\theta|^2\Gamma_2(\by_1,\bphi)+(\theta^\top\Gamma_2(\by_1,\bphi))\theta\Big)\bnu_2(d\bphi)\\
\aad\ +o(1)\sqrt{V(\by_1)},\eea
\bea
g_2(\by_1,\theta,r)\ad =\frac{\sigma_2(\bz)}{|\by_2|}-\frac{\by_2\by_2^\top\sigma_2(\bz)}{|\by_2|^3}=(\Sigma_{21}(\by_1)\theta,\dots,\Sigma_{2l}(\by_1)\theta)-\theta\theta^\top(\Sigma_{1l}(\by_1)\theta,\dots,\Sigma_{2l}(\by_1)\theta),\\
g_3(\by_1,\theta,r,\bphi)&=\frac{\by_2+\gamma_2(\bz,\bphi)}{|\by_2+\gamma_2(\bz,\bphi)|}-\frac{\by_2}{|\by_2|}=\frac{\theta+\Gamma_2(\by_1,\bphi)\theta}{|\theta+\Gamma_2(\by_1,\bphi)\theta|}-\theta,\\
\eea
\bea
h_1(\by_1,\theta,r)\ad= 2\by_2 b_2(\bz) +\trace(\sigma_2^\top(\bz)\sigma_2(\bz))+\int_{\Rb} |\gamma_2(\bz,\bphi)|^2\bnu_2(d\bphi)\\
\ad =r\Big(2\theta^\top B_2(\by_1)\theta+\sum_{l=1}^{d_2}|\Sigma_{2l}(\by_1)\theta|^2+\int_{\Rb}|\Gamma_2(\by_1,\bphi)\theta|^2\bnu_2(d\bphi)\Big)+ ro(r) V_1(\by_1),\\
h_2(\by_1,\theta,r)\ad = 2\by_2^\top \sigma_2(\bz)=2r\sum_{l=1}^{d_2}\theta^\top\Sigma_{2l}(\by_1)\theta+ ro(r) \sqrt{V_1(\by_1)},\\
h_3(\by_1,\theta,r,\bphi)\ad =|\by_2+\gamma_2(\bz,\bphi)|^2-|\by_2|^2=r(|\theta+\Gamma_2(\by_1,\bphi)\theta|^2-1).
\eea
Let $\BX_1(t)=([\BY_1(t)]^\top,[\Theta(t)]^\top)^\top\in \RX\times\Sc_d$ and $\BX_2(t)=R(t)\in\R_+$. We have that
$$
\begin{aligned}
\ln R(t)=& h_4(\BY_1(t),\Theta(t),  R(t))dt + h_5(\BY_1(t),\Theta(t),  R(t))d\BW_2(t)\\
&+\int_{\Rb}h_6(\BY_1(t),\Theta(t),  R(t),\bphi)\wN_2(dt,d\bphi),
\end{aligned}
$$
where
$$
\begin{aligned}
h_4(\by_1,\theta,r)=&2\theta^\top B_2(\by_1)\theta+\sum_{l=1}^{d_2}|\Sigma_{2l}(\by_1)\theta|^2-2|\theta^\top\Sigma_{2l}(\by_1)\theta|^2\\
&+\int_{\Rb}\Big(\ln|\theta+\Gamma_2(\by_1,\bphi)\theta|^2-|\theta+\Gamma_2(\by_1,\bphi)\theta|^2+1\Big)\bnu_2(d\bphi)+o(1)V(\by_1),\\
h_5(\by_1,\theta,r)=&\frac{2\by_2^\top \sigma_2(\bz)}{|\by_2|^2}=2\sum_{l=1}^{d_2}\theta^\top\Sigma_{2l}(\by_1)\theta +o(1)\sqrt{V(\by_1)},\\
h_6(\by_1,\theta,r,\bphi)&=\ln\frac{|\by_2+\gamma_2(\bz,\bphi)|^2}{|\by_2|^2}=\ln |\theta+\Gamma_2(\by_1,\bphi)\theta|^2 +o(1)\sqrt{V(\by_1)}.
\end{aligned}
$$
When $r=0$, the equation for $\BX_1(t):=([\BY_1(t)]^\top,[\Theta(t)]^\top)^\top$ is
\begin{equation}\label{e4.3}
\begin{cases}
\displaystyle
d\BY_1(t)=b_1(\BY_1(t),\0)dt+\sigma_1(\BY_1(t),\0)d\BW_1(t)+\int_\Ra \gamma_1(\BY_1(t-),\0, \bphi)\wN_1(dt,d\bphi), \\
\displaystyle
d\Theta(t)=g_1(\BY_1(t),\Theta(t),  0)dt+g_2(\BY_1(t),\Theta(t),0)d\BW_2(t)+\int_{\Rb}g_3(\BY_1(t), \Theta(t), 0,\bphi)\wN_2(dt,d\bphi).
\end{cases}
\end{equation}
Assume that $\bnu_1$ and $\bnu_2$ are finite measures.
Then
the system has a unique invariant measure on $\RX\times\Sc_d$ denoted by $\Pi$, due to the non-degeneracy of the diffusion, which follows from the explicit formula for $g_i$ given above, and the boundedness of $\Theta(t)$ and the Assumption \ref{asp2}(iv).

Now, let $U(r)=(-\ln|r|)\vee 0$.
Then we can apply Theorem \ref{thm-main} to show that if
$$
\begin{aligned}
\int_{\RX\times\Sc_d}\Big(2\theta^\top B_2(\by_1)\theta&+\sum_{l=1}^d\big(|\Sigma_{2l}(\by_1)\theta|^2-2|\theta^\top\Sigma_{2l}(\by_1)\theta|^2\big)\\
&+\ln|\theta+\Gamma_2(\by_1,\bphi)\theta|^2-|\theta+\Gamma_2(\by_1,\bphi)\theta|^2+1\Big) \Pi(d\by_1,d\theta)<0,
\end{aligned}
$$
then $R(t)$ converges to 0 in probability.
Note that the component $\Theta$ lives in a compact manifold $\Sc_d$ but Theorem \ref{thm-main} can be applied here
because the coupling method in Lemma \ref{lm3} and Theorem \ref{thm-main} can be done through the $(d-1)$-dimensional Euclidean coordinates of $\Sc_d$ and
the right inverse of the diffusion coefficient for $(\BY_1(t),\Theta(t))$ is bounded.

%
%
%
%
\subsection{Stability of the Slow Component in a Fast-Slow System}
%

We study the fast-slow coupled jump diffusions as follows
\begin{equation}\label{e1-2.2}
{\small
\begin{cases}
\displaystyle
d\BY_1^\eps(t)=\frac1{\eps} b_1(\BY_1^\eps(t),\BY_2^\eps(t))dt+\frac1{\sqrt{\eps}}\sigma_1(\BY_1^\eps(t),\BY_2^\eps(t))d\BW_1(t)
+\int_\Ra \gamma_1(\BY_1^\eps(t-),\BY_2^\eps(t-), \bphi)\wN_1^{\eps}(dt,d\bphi),
\\
\displaystyle  d\BY_2^\eps(t)=b_2(\BY_1^\eps(t),\BY_2^\eps(t))dt+\sigma_2(\BY_2^\eps(t))d\BW_2(t)+\int_\Rb \gamma_2(\BY_2^\eps(t-), \bphi)\wN_2(dt,d\bphi),\end{cases}}
\end{equation}
where $\wN_1^{\eps}(dt,d\phi)=\bold N_1(dt,d\phi)-\frac 1{\eps}\bnu_1(d\bphi)dt$,
and $\sigma_2$ and $\gamma_2$ are assumed to be functions of $\BY_2^\eps$ only.
We will study the stability of the slow component for small
$\eps$
based
on the stability of the limit system.
This problem is very important in applications. Typically, the limit system is often much easier to analyze and compute.

We assume that Assumption \ref{asp3.1} holds
with $\Sigma_{21},\dots,\Sigma_{2d_2}$ and $\Gamma_2(\bphi)$ are now independent of $\by_1$. Let $g_i,h_i$ be functions defined as in Section \ref{sec:app1}.
Using the change of variable as in Section \ref{sec:app1},
by the generalized It\^o
formula for jump diffusions,
we have the following equations for $\BY_1(t)$, $\Theta^\eps(t)=\frac{\BY_2^\eps(t)}{|\BY_2^\eps(t)|}$, $R^\eps(t)=|\BY_2^\eps(t)|^2$
\begin{equation}\label{e2-2.2}
{\small
\begin{cases}
d\BY_1^\eps(t)=&\!\!\!\displaystyle\frac1{\eps} b_1(\BZ^\eps(t))dt+\frac1{\sqrt{\eps}}\sigma_1(\BZ^\eps(t))d\BW_1(t)+\int_\Ra \gamma_1(\BZ^\eps(t-), \bphi)\wN_1^{\eps}(dt,d\bphi),\\
 d\Theta^\eps(t)=&\!\!\! g_1(\BY^\eps_1(t),\Theta^\eps(t),  R^\eps(t))dt+g_2(\Theta^\eps(t),  R^\eps(t))d\BW_2(t)
 +\displaystyle\int_{\Rb}g_3(\Theta^\eps(t),  R^\eps(t),\bphi)\wN_2(dt,d\bphi),\\
 d R^\eps(t)=&\!\!\! h_1(\BY^\eps_1(t),\Theta^\eps(t),  R^\eps(t))dt+h_2(\Theta^\eps(t),  R^\eps(t))d\BW_2(t)
 +\displaystyle\int_{\Rb}h_3(\Theta^\eps(t),  R^\eps(t),\bphi)\wN_2(dt,d\bphi),\end{cases}}
\end{equation}
where $\BZ^\eps(t)=([\BY_1^\eps(t)]^\top,[\BY_2^\eps(t)]^\top)^\top=([\BY_1^\eps(t)]^\top,\sqrt{R^\eps(t)}[\Theta^\eps(t)]^\top)^\top$.

Let $\Pi^\eps$ be the family of invariant measures of the system
\begin{equation}\label{e3-2.2}
\begin{cases}
\displaystyle d\BY_1^\eps(t)=\frac1{\eps} b_1(\BY_1^\eps,\0)dt+\frac1{\sqrt{\eps}}\sigma_1(\BY_1^\eps(t),\0)d\BW_1(t)+\int_\Ra \gamma_1(\BY_1^\eps(t), \0, \bphi)\wN_1^{\eps}(dt,d\bphi),\\
\displaystyle d\Theta^\eps(t)=g_1(\BY_1^\eps(t),\Theta^\eps(t), 0)dt+g_2(\Theta^\eps(t), 0)d\BW_2(t)+\int_{\Rb}g_3(\Theta^\eps(t),0,\bphi)\wN_2(dt,d\bphi).
\end{cases}
\end{equation}
By a standard averaging principle (see e.g., \cite{XL20}), we can show that
	$\Pi^\eps$ converges weakly to $\Pi_1\times\Pi_2$
where $\Pi_1$ is the invariant measure of the system (due to the fast component of \eqref{e3-2.2} is decoupled from the slow component),
$$
\displaystyle d\BY_1(t)=b_1(\BY_1,\0)dt+\sigma_1(\BY_1,\0)d\BW_1(t)+\int_\Ra \gamma_1(\BY_1(t), \0, \bphi)\wN_1(dt,d\bphi),
$$
and
$\Pi_2$  is the invariant measure of the averaged system
$$
d\Theta(t)=\bar g_1(\Theta(t))dt+ g_2(\Theta(t), 0)dW_2(t)+\int_{\Rb} g_3(\Theta(t),0,\bphi)\wN_2(dt,d\bphi),
$$
with
$$
\begin{aligned}
\bar g_1(\theta)=&\int_{\RX} g_1(\by_1,\theta,0)\Pi_1(d\by_1)\\
=&\Big(-\theta^\top\bar B_2\theta +\frac12\sum_{l=1}^d\big[-|\Sigma_{2l}\theta|^2+3|\theta^\top\Sigma_{2l}\theta|^2\big]\Big)\theta\\
&+\int_{\Rb}\Big(\frac{\theta+\Gamma_2(\bphi)\theta}{|\theta+\Gamma_2(\bphi)\theta|}-\theta-|\theta|^2\Gamma_2(\bphi)+(\theta^\top\Gamma_2(\bphi))\theta\Big)\bnu_2(d\bphi),
\end{aligned}
$$
where
$$\bar B_2=\int_{\RX} B_2(\by_1)\Pi_1(d\by_1).$$
In view of Section \ref{sec:app1}, the condition for the stability of $\BY_2^\eps$ is 
$$
\begin{aligned}
\lambda_\eps:=\int_{\RX\times\Sc_d}\Big(2\theta^\top B_2(\by_1)\theta&+\sum_{l=1}^d\big(|\Sigma_{2l}\theta|^2-2|
\theta^\top\Sigma_{2l}\theta|^2\big)\\
&+\ln|\theta+\Gamma_2(\bphi)\theta|^2-|\theta+\Gamma_2
(\bphi)\theta|^2+1\Big) \Pi^\eps(d\by_1,d\theta)<0.
\end{aligned}
$$
Since $\Pi^\eps(d\by_1,d\theta)$ converges weakly to $\Pi_1\times\Pi_2$ as $\eps\to\infty$, we have
$$
\begin{aligned}
\lim_{\eps\to0}\lambda_\eps=\lambda_*:=&\int_{\RX\times\Sc_d}\Big(2\theta^\top B_2(\by_1)\theta+\sum_{l=1}^d\big(|\Sigma_{2l}\theta|^2-2|\theta^\top\Sigma_{2l}\theta|^2\big)\\
&\qquad+\ln|\theta+\Gamma_2(\bphi)\theta|^2-|\theta
+\Gamma_2(\bphi)\theta|^2+1\Big) (\Pi_1\times\Pi_2)(d\by_1,d\theta).
\end{aligned}
$$
In view of \cite{LDS02}, it is easy to check that $\lambda_*<0$ is the   necessary and sufficient condition for
 the following linear system to be exponentially stable
\begin{equation}\label{eq-3-lin}
d\bar\BY_2(t)=\bar B_2\bar\BY_2(t)dt+\sum_{l=1}^d\Sigma_{2d}\bar\BY_2(t)d\BW_{2l}(t)+\int_\Ra \Gamma_2(\bphi)\bar\BY_2(t)\wN_2(dt,d\bphi).
\end{equation}
In fact,
$$\lim_{t\to\infty}\frac{\ln |\bar \BY_2(t)|}t=\lambda_*<0.$$
Thus, if $\lambda_*<0$ then
$$\lim_{\by_2\to0}\PP_{\by_1,\by_2}\left\{\lim_{t\to\infty}\frac{\ln |\BY_2^\eps(t)|}t<\frac{\lambda_*}2<0\right\}=1,$$
as $\eps$ is sufficiently small.
So, we can obtain the condition for the stability of fast-slow coupled jump-diffusion systems.

\begin{rem}
Treating two time-scale systems, one often  uses the so-called freezing component argument; see \cite{Kh68} and \cite[pp.88-90]{Kushner90}.
Here we use a somewhat different argument. Using polar decomposition \eqref{e3-2.2}, whether $R^\eps$ converges to 0 depends on the invariant probability measure $\Pi^\eps$ of $(Y_1^\eps, \Theta^\eps)$ for each $\eps$. 
Then the Lyapunov exponent that determines stability and that is given by $\lambda^\eps$, is computed based on $\Pi^\eps$. Finally, we show $\lambda^\eps$ converges to $\lambda^*$, which is obtained based on the limit system.
\end{rem}

\section{Stabilization and Consensus Problems}\label{sec:stabl}
\subsection{Stabilization}
In this section, we consider the controlled jump-diffusion system given by following equations,
\begin{equation}\label{e1-2.3}
{\small
\begin{cases}
 d\BX_1(t)=&\disp b_1(\BX_1(t),\BX_2(t))dt+\sigma_1(\BX_1(t),\BX_2(t))d\BW_1(t)+\int_\Ra \gamma_1(\BX_1(t-),\BX_2(t-), \bphi)\wN_1(dt,d\bphi), \\
&+u(t)dt\\
 d\BX_2(t)=&\disp b_2(\BX_1(t),\BX_2(t))dt+\sigma_2(\BX_1(t),\BX_2(t))d\BW_2(t)+\int_\Rb \gamma_2(\BX_1(t-),\BX_2(t-), \bphi)\wN_2(dt,d\bphi),\\
\BX_1(0)=&\disp \bx_1,\quad \BX_2(0)=\bx_2,
\end{cases}}
\end{equation}
where $u(t)$ is a control.
We want to construct a control so as to stabilize the
process $\BX_2$. However, we
cannot
act directly to $\BX_2$
but only the interacting process $\BX_1$ can be controlled.
We will apply our result
to show that under certain conditions, we can control the interacting process $\BX_1$ to have the stability of the
process $\BX_2$.
A system may or may not have an invariant probability measure. The weak stabilization essentially means that we construct a control so that the resulting system is weakly stable. That is, the resulting system has an invariance measure. The term weak stability was originated from the work of Wonham \cite{Wonham}.

\begin{asp}{\rm
	There exists a function $U:\RY\mapsto[0,\infty)$ such that
$$\lim_{\bx_2\to\0} U(\bx_2)=\infty,\quad  \sup_{|\bx_2|\geq\deltaxe}U(\bx_2)<\infty, \quad U(\bx_2) -U(\bx_1)\leq c_0\ln\frac{|\bx_1|}{|\bx_2|};$$
and there are a constant $\Delta_0$ and functions $f_1,f_2:\RX\to\R$ so that
$$[\op U](\bz)\leq f_1(\bx_1),\quad \left(U_{\bx_2}\sigma_2(\bz)\right)^2\leq f_2(\bx_1),\quad\forall  |\bx_2|\leq \Delta_0.$$
Moreover, we suppose that $f_1$ and $f_2$ are bounded above by $K(1+|\bx_1|^2)$
and
$f_1(\0)<0$ and $\lim_{\bx_1\to\infty}f_1(\bx_1)>0$.
Finally, suppose that there is a matrix $Q$ such that
$$
b_1^\top(\bz)Q+\trace{\sigma_1^\top(\bz)Q\sigma_1(\bz)}+
\int_\Ra \gamma_1^\top(\bz, \bphi)Q\gamma_1(\bz, \bphi)\bnu_1(d\bphi)\leq c_1+c_2|\bx_1|, \text{ if } |\bx_2|\leq \Delta_0.
$$}
\end{asp}

From the assumption on $f_1$, we can write $f_1(\bx_1)$ as
$f_1(\bx_1)\leq -K_1+K_2|\bx_1|^2$ for some constants $K_1, K_2$.
Consider the  control $u(t)=A\BX_1(t)$,
where $A$ is a matrix satisfying $$-\lambda_A:=\max_{\bx_1\in\RX}\frac{\bx_1^\top QA\bx_1}{|\bx_1|^2}<
-\frac{K_2 c_1 + K_1c_2}{K_1}.
$$	
Now, when $\BX_2=\0$ the corresponding system for $\BX_1$ is
$$
d\BX_1(t)=\Big(b_1(\BX_1(t),\0)+u(t)\Big)dt+\sigma_1(\BX_1(t),\0) dW_1(t)+\int_\Ra \gamma_1(\BX_1(t),\0, \bphi)\wN_1(dt,d\bphi).
$$
We have for $V(\bx_1):=\bx_1^\top Q\bx_1$ that
$$
\begin{aligned}
\text{}\op V(\bx_1)=& b_1^\top(\bx,\0)Q+\trace{\sigma_1^\top(\bx_1,\0)Q\sigma_1(\bx_1,\0)}+
\int_\Ra \gamma_1^\top(\bx_1,\0, \bphi)Q\gamma_1(\bx_1,\0, \bphi)\bnu_1(d\bphi)+\bx_1^\top QA\bx_1\\
\leq & c_1-(\lambda_A-c_2)|\bx_1|^2.
\end{aligned}
$$
As a result, when $\bx_2=\0$, there exists a unique invariant measure $\Pi_1$ for $\BX_1(t)$ and
$$
c_1-(\lambda_A-c_2)\int_{\RX}|\bx_1|^2\Pi_1(d\bx_1)\geq 0.
$$
That yields
$$
\int_{\RX}f_1(\bx_1)\Pi_1(d\bx_1)\leq -K_1+K_2\frac{c_1}{\lambda_A-c_2}<0,
$$
which implies the stability of the controlled system by an application of our main result (Theorem \ref{thm-main}).

\subsection{
Leader-Following Consensus Problems}
In this section, we apply our results to the leader-following consensus problems.
We consider a network with a leader and $N$ identical followers.
The dynamics of the leader is described by
\begin{equation}\label{eq-x0}
d\bx_0(t)=f(\bx_0(t))dt+d\BW(t),
\end{equation}
and the
dynamics of the $i$-th follower is described by
\def\bu{\bold u}
\begin{equation}\label{eq-y}
d\bx_i(t)=f(\bx_i(t))dt+B\bu_i(t)dt+d\BW(t), \quad i=1,\dots,N,
\end{equation}
where $\bx_i(t)\in\R^n$, $i=0,\dots,N$,
$f:\R^n\to\R^n$, 
$\BW(t)$ is an $n$-dimensional Brownian motion,
$\bu=[\bu_1^\top,\dots,\bu_N^\top]^\top$ ($\bu_i\in\R^n,i=1,\dots,N$) is the control to be designed,
$B\in\R^{n\times n}$.

Now, we model the information flow structure among different agents as follows. Let $\mathcal {G}=\{\mathcal V,\mathcal {E}, \A\}$ be a connected directed graph, where
$\mathcal V=\{0,\dots,N\}$ denotes the set of nodes with $k$ representing the $k^{th}$ agent,
$\mathcal{E}$ is the set of directed edges,
$\A=[a_{kl}]\in\R^{N\times N}$ is the adjacency matrix of $\mathcal{G}$ with
$a_{kl}=1$ or $0$ indicating whether or not there is a directed information flow between agents $l$ and
$k$.
Also, denoted by $\mathcal N_k$ the set of the node $k$'s neighbors, i.e., $\mathcal N_k:=\{l\in \{1,\dots,N\}: a_{kl}=1\}$,
and
deg$_k:=\sum_{l=1}^N a_{kl}$ the degree of $k$-th agent.
The Laplacian matrix of $\mathcal{G}$ is defined as $\h=\D-\A$, where $\D=\diag (\text{deg}_1,\dots,\text{deg}_N).$

Suppose that the leader may send information to the followers, but does not receive information from any one of the
followers,
and that $G$
contains a spanning tree rooted at the leader;
and that
the communication between followers are undirected.
Then we obtain the Laplacian matrix $\h$ of $G$ as follows
$$
\h=\left[
\begin{array}{c c}
0& \bold 0_N^\top\\
\bold a_0& \wdt\h
\end{array}
\right],
$$
where $\bold a_0=(a_{10},\dots,a_{N0})^\top$,
$\bold 0_N:=(0,\dots,0)^\top\in\R^N$ and
$$
\wdt \h=\left[
\begin{array}{c c c c}
\displaystyle\sum_{j=0,j\neq 1}^Na_{1j}& -a_{12}&$\dots$ &-a_{1N}\\
-a_{21}&\displaystyle\sum_{j=0,j\neq 2}^N a_{2j}&$\dots$&-a_{2N}\\
$\vdots$&$\vdots$&$\vdots$&$\vdots$\\
-a_{N1}&-a_{N2}&\cdots&\displaystyle\sum_{j=0,j\neq N}^N a_{Nj}
\end{array}
\right].
$$
Since we assume that the communication between the followers is undirected
in this paper (for simplicity of notation),
$\wdt\h$ is symmetric
and the eigenvalues of $\wdt\h$ are real numbers.

For the $i$-th follower, we consider the following leader-following
consensus protocol
\begin{equation}\label{eq-u}
\bu_i(t)=K\sum_{j=0,j\in\mathcal N_i}^N \bz_{ji}(t), i=1,\dots,N,
\end{equation}
where the symmetric matrix $K\in \R^{n\times n}$ is the control gain to be designed, and
$$\bz_{ji}(t)=\bx_j(t)-\bx_i(t)+(\bx_i(t)-\bx_j(t))\xi_{ji}(t),$$
 is the measurement of the agent $j$ from its neighbor agent $i$,
and $\xi_{ji}$'s are some random noises.
Denoted by
$$\mathcal{U}=\{\bu(t)=([\bu_1(t)]^\top,\dots,[\bu_N(t)]^\top)^\top| \bu_i(t) \text{ is given by \eqref{eq-u}}, t\geq 0,\text{ and }i=1,\dots,N\},$$
the collection of all admissible distributed protocols.
We refer the reader to \cite{NWX13,YHY19} and references therein for 
motivation of the above system.

\begin{asp}\label{asp-3.3} {\rm
\begin{itemize} We assume the following:
\item[{\rm(i)}] $f(y)$ satisfies $|f(y)|\leq -c|y|$, for all $y\in\R^n$ for some constant $c>0$.
\item[{\rm(ii)}] The noise $\xi_{ji}(t)$ satisfies that
$$
\int_0^t (\bx_j(s)-\bx_i(s))\xi_{ji}(s)ds=\int_0^t (\bx_j(s)-\bx_i(s))\left(\sigma_{ji}dw_{ij}(s)+\int_\Ra \gamma_{ji}(\bphi)\wN_{ji}(ds,d\bphi)\right),
$$
where
$w_{ji}(s)$ are independent standard Brownian motions,
$\wN_{ji}(s,\phi)$ are jump processes.
\end{itemize}
}\end{asp}

\begin{deff}
{\rm
System \eqref{eq-x0} and \eqref{eq-y} is said to be exponentially consentable in probability with respect to $\mathcal{U}$ if there exists a protocol $\bu\in\mathcal U$ so that for any $\eps>0$, there exists $\delta>0$ such that
for all $i=1,\dots,N$
$$\PP\left\{|\bx_i(t)-\bx_0(t)|\text{ converges exponentially fast to } 0\right\}\geq 1-\eps,$$
whenever the initial values $(\by_0,\dots,\by_n)\in\R^{nN}$ of \eqref{eq-x0} and \eqref{eq-y} satisfying that
$$
\sum_{i=1}^N|\by_i-\by_0|^2<\delta.
$$
}\end{deff}
To proceed,
let $\BX_i(t)=\bx_i(t)-\bx_0(t), i=1,\dots,N$,
$\bx(t):=\left[\bx_1^\top(t),\dots,\bx_N^\top(t)\right]^\top$,
$F_i(\bx_0(t),\bx_i(t)):=f(\bx_i(t))-f(\bx_0(t))$,
and
$F(\bx_0(t),\dots,\bx_N(t)):=\left[F^\top_1(\bx_0(t),\bx_1(t)),\dots,F^\top_N(\bx_0(t),\bx_N(t))\right]^\top$,
$\BX(t):=\left[\BX_1^\top(t),\dots,\BX_N^\top(t)\right]^\top$.
For simplicity of notation, we will write $F(\bx_0(t),\dots,\bx_N(t))$ as $F(\bx_0(t),\BX(t))$ by the identity
$F(\bx_0(t),\BX(t))=F(\bx_0(t),\bx_0(t)+\BX_1(t),\dots,\bx_0(t)+\BX_N(t))$.
Then, one can obtain
\begin{equation}\label{eq-2.4-delta}
\begin{cases}
d\bx_0(t)=f(\bx_0(t))dt+d\BW(t),\\
d\BX(t)=\left(F(\bx_0(t),\BX(t))-\wdt \h \otimes BK\delta(t)\right)dt+dM_1(t)+dM_0(t)+dM^N_0(t)+dM^N_1(t),
\end{cases}
\end{equation}
where $\mathbf A\otimes \mathbf B$ denotes the Kronecker product  of $\mathbf A$ and $\mathbf B$, and
$$
\begin{aligned}
&M_1(t)=\sum_{i,j=1}^N\int_0^t \sigma_{ji}\left[S_{ij}\otimes BK\right]\BX(s)dw_{ji}(s),\\
&M_0(t)=-\sum_{i=1}^N\int_0^t\sigma_{0i}\left[\bar S_i\otimes BK\right]\BX(s)dw_{0i}(s),\\
&M^N_1(t)=\sum_{i,j=1}^N\int_0^t \int_\Ra \left[S_{ij}\otimes BK\right]\BX(s)\gamma_{ji}(\bphi)\wN_{ji}(ds,d\bphi),\\
&M^N_0(t)=-\sum_{i=1}^N\int_0^t\int_\Ra\left[\bar S_i\otimes BK\right]\BX(s)\gamma_{ji}(\bphi)\wN_{ji}(ds,d\bphi),
\end{aligned}
$$
where
$S_{ij}=\left[s_{kl}\right]_{N\times N}$ is an $N\times N$ matrix with $s_{ii}=-a_{ij}$ and $s_{ij}=a_{ij}$ and all other elements being $0$, for $i,j=1,\dots,N$,
and
$\bar S_i=\left[\bar s_{kl}\right]_{N\times N}$ is an $N\times N$ matrix with $\bar s_{ii}=a_{i0}$ and all other entries being $0$.
It is easily seen that the consensus problem of \eqref{eq-x0} and \eqref{eq-y} is equivalent to the stability of \eqref{eq-2.4-delta}.

\begin{asp}\label{asp-leader}{\rm
There exists a matrix $K\in\R^{n\times n}$ such that there exists a function $U:\R^{nN}\mapsto[0,\infty)$ satisfying the following conditions,
$$\lim_{\BX\to\0} U(\BX)=\infty,\quad U(\BX) -U(\BX')\leq c_0\ln\frac{|\BX'|}{|\BX|},$$
and there is $\Delta_0$ such that
$$[\op U](\BX)\leq c_1(\bx_0),|\BX|\leq \Delta_0,$$
$$ \sum_{i,j=1}^N\left(U_{\BX}^\top(\sigma_{ji}\left[S_{ij}\otimes BK\right]\BX)\right)^2+\sum_{i=1}^N\left(U_\BX^\top( \sigma_{0i}\left[\bar S_i\otimes BK\right]\BX)\right)^2\leq c_2(\bx_0), |\BX|\leq \Delta_0,$$
$$\int_\Rb\Big[\exp\Big\{-\alpha_0\big(U(\BX_2+\gamma_2(\BX,\bphi))-U(\BX)\big)_+\Big\}\Big]\bnu_2(\bphi)\leq c_3(\bx_0), |\BX|\leq \Delta_0,
$$
and
$$
\int \left(c_1(\bx_0)+c_2(\bx_0)+c_3(\bx_0)\right)\mu^*(dx_0)>0.
$$
In above, $\mu^*(\cdot)$ is the invariant measure of \eqref{eq-x0}. Such a $\mu^*(\cdot)$  always exists because of Assumption \ref{asp-3.3}.}
\end{asp}

\begin{thm}
Under Assumptions {\rm \ref{asp-3.3}} and {\rm\ref{asp-leader}}, system \eqref{eq-2.4-delta} is exponentially stable in probability. As a consequence, the leader-following system \eqref{eq-x0} and \eqref{eq-y} is
exponentially consentable in probability.
\end{thm}

In fact, our above results are verifiable. To illustrate that, we 
provide the following explicitly computational example.

\begin{exam}{\rm
In this example, assume that $f(x)=Ax$ where $A\in\R^{n\times n}$.
Assume that $B$ is invertible.
Let $U(\delta)=-\ln|\delta|$, 
by directed calculations, we have
$$
\Lom U(\BX)=-\frac{\BX \left(A\otimes I_N-\wdt \h \otimes BK\right)\delta^\top}{|\BX|^2}.
$$
Then, it is easy to check the remaining conditions.
}\end{exam}

\section{Conclusion}\label{sec:con}
We studied stability and stabilization of a fully coupled system of jump diffusions. Sufficient conditions for stability
are derived. 
We then  investigate the stability of linearizable jump diffusions and fast-slow coupled jump diffusions. Next, we develop strategies for weak stabilization of
a coupled system in which only one component can be controlled. Also considered are
consensus problem of leader-following systems.
The consideration of this paper can be readily extended to systems with many components. There are many interesting important problems remain to be investigated.
Future research could be extended
for regime-switching with state-dependent diffusions or hidden Markov systems. Efforts can also be directed to studying systems with mean-fields interactions. These and other topics deserve to be carefully examined.

\end{document}